\documentclass[a4paper,10pt]{article}

\newenvironment{proof}{\noindent {\bf Proof:}}{\hfill $\Box$}

\newtheorem{theorem}{Theorem}
\newtheorem{lemma}{Lemma}

\newtheorem{corollary}{Corollary}

\newtheorem{definition}{Definition}
\newtheorem{assumption}{Assumption}
\newtheorem{remark}{Remark}

\usepackage{booktabs}
\usepackage{color}

\usepackage{algorithm}
\usepackage[noend]{algpseudocode}
\makeatletter
\def\BState{\State\hskip-\ALG@thistlm}
\makeatother
\usepackage{algorithmicx}

\usepackage{amsmath}
\usepackage{amssymb}
\usepackage{amsfonts}
\usepackage{hyperref}

\usepackage[normalem]{ulem}
\usepackage[dvipsnames]{xcolor}

\newcommand{\new}[1]{{\color{black}#1}}

\usepackage{tikz}

\newcommand{\bs}{\boldsymbol}
\newcommand{\mr}[1]{\mathrm{#1}}

\newcommand{\B}{\mathcal{B}}
\newcommand{\C}{\mathcal{C}}

\newcommand{\Mc}{\mathcal{M}}

\newcommand{\Bc}{\mathcal{B}}

\newcommand{\Vc}{\mathcal{V}}

\newcommand{\Uf}{\mathbf{U}}

\newcommand{\Xf}{\mathbf{X}}
\newcommand{\Af}{\mathbf{A}}

\newcommand{\zf}{\mathbf{z}}
\newcommand{\cf}{\mathbf{c}}
\newcommand{\ones}{\mathbf{1}}

\newcommand{\Rb}{\mathbb{R}}
\newcommand{\Nb}{\mathbb{N}}

\newcommand{\argmin}{\operatornamewithlimits{arg\,min}}

\newcommand{\dist}{\mathrm{dist}}
\newcommand{\proj}{\mathrm{proj}}
\newcommand{\vol}{\mathrm{vol}}
\newcommand{\refMeas}{\lambda}

\usepackage[symbol]{footmisc}

\renewcommand{\thefootnote}{\fnsymbol{footnote}}

\title{\bf Computing controlled invariant sets from data using convex optimization}

\begin{document}

\author{Milan Korda$^{\ast\,\dagger}$}

\footnotetext[1]{CNRS; LAAS; 7 avenue du colonel Roche, F-31400 Toulouse; France. {\tt korda@laas.fr}}
\footnotetext[2]{Faculty of Electrical Engineering, Czech Technical University in Prague,
Technick\'a 2, CZ-16626 Prague, Czech Republic.}
\footnotetext[3]{This work has been supported by European Union’s Horizon 2020 research and innovation programme under the Marie Skłodowska-Curie Actions, grant agreement 813211 (POEMA) and by the Czech Science Foundation
(GACR) under contract No. 20-11626Y.}

\renewcommand*{\thefootnote}{\arabic{footnote}}

\date{ \small \today}

\maketitle

\begin{abstract}
This work presents a data-driven method for approximation of the maximum positively invariant (MPI) set and the maximum controlled invariant (MCI) set for nonlinear dynamical systems. The method only requires the knowledge of a finite collection of one-step transitions of the discrete-time dynamics, without the requirement of segments of trajectories or the control inputs that effected the transitions to be available. The approach uses a novel characterization of the MPI and MCI sets as the solution to an infinite-dimensional linear programming (LP) problem in the space of continuous functions, with the optimum being attained by a (Lipschitz) continuous function under mild assumptions. The infinite-dimensional LP is then approximated by restricting the decision variable to a finite-dimensional subspace and by imposing the non-negativity constraint of this LP only on the available data samples. This leads to a single finite-dimensional LP that can be easily solved using off-the-shelf solvers. We analyze the convergence rate and sample complexity, proving probabilistic as well as hard guarantees on the volume error of the approximations. The approach is very general, requiring minimal underlying assumptions. In particular, the dynamics is not required to be polynomial or even continuous (forgoing some of the theoretical results). Detailed numerical examples up to state-space dimension ten with code available online demonstrate the method\footnote{Matlab code available at \url{https://homepages.laas.fr/mkorda/MCI_data_driven.zip}}.
\end{abstract}

\begin{flushleft}\small
\end{flushleft}

\section{Introduction}
The maximum controlled invariant (MCI) set associated to a given controlled dynamical system and a given state and input constraint set is the set of all initial conditions that can be kept within the state constraint set forever using admissible control inputs. If there is no control, this set is referred to as the maximum positively invariant (MPI) set. Both sets are fundamental objects associated to any (controlled) dynamical system, determining the limitations of the system in terms of constraint satisfaction over an infinite time horizon. In terms of applications, the knowledge of these sets is critical for safety as well as for controller synthesis~\cite{Blanchini_paper}; for example, determining a control invariant set is a key element for recursively feasible model predictive control design~\cite{gondhalekar_controlled_2009, mayne_constrained_2000}.

The computation of these sets has been the subject of intensive research for the last decades. Let us mention in particular the classical contractive algorithm, described, e.g., in the monograph~\cite{BlanchiniBook}, the expansive algorithm of~\cite{gutman1987algorithm} or the more recently proposed approach based on occupation measures and convex semidefinite programming (SDP)~\cite{kordaMCI}. However, to the best of our knowledge, all of the algorithms proposed in the literature are model-based, i.e., they require the knowledge of the transition mapping $f$ of the discrete-time dynamical system. \new{On the other hand, the recent work~\cite{kenanian2019data} works in a data-driven setting but addresses a different problem, that of deciding the stability of an uncontrolled switched-affine system.}

In this work, we develop a fully \emph{data-driven} algorithm, requiring only the knowledge of a finite collection of one-step transitions $\left\{(x_i, x_i^+)\right\}_{i=1}^K$ with $x_i^+ = f(x_i)$ or $x_i^+ = f(x_i,u_i)$. In particular, in the controlled setting, we do not require the knowledge of the control inputs that effected the transitions.  We also do not require the knowledge of entire trajectories, or sufficiently long segments thereof. Our approach is based on sample-based discretization of a novel characterization of the MPI or MCI set as the solution to an infinite-dimensional linear programming problem in the space of continuous functions. Contrary to existing LP characterizations of these sets (e.g., \cite{kordaMCI}), the optimizer of this LP is attained in the space of (Lipschitz) continuous functions, while still being amenable for data-driven discretization. The attainment of the optimizer is crucial for subsequent theoretical analysis of the convergence of finite-dimensional approximations to this LP as well as to the sample complexity analysis when using data-driven discretization, both of which are carried out in this paper.

The contribution can be summarized as follows:
\begin{itemize}
\item We provide a novel characterization of the MCI and MPI sets as the solution to an infinite-dimensional LP in the space of continuous functions whose optimizer is attained, under mild regularity assumptions.
\item We analyze the convergence rate of approximating the LP by restricting the space of decision variables to a finite-dimensional subspace.
\item We propose a data-driven discretization scheme of the LP and derive high-probability bounds on the volume error of this discretization in terms of the number of samples and the dimension of the approximating subspace. We also discuss how hard guarantees can be obtained with additional modeling information.
\end{itemize}

The main characteristics of the proposed approach are:
\begin{itemize}
\item The algorithm is data-driven, requiring no explicit knowledge of the underlying dynamics. In particular, contrary to~\cite{kordaMCI}, the dynamics does not need to be polynomial or even continuous (forgoing some of the theoretical results); in particular, the approach is applicable to nonlinear switched and hybrid systems.
\item The algorithm is based on the solution to a single finite-dimensional linear programming problem, with no iteration or complicated initialization involved. The scalability appears to be superior to the SDP-based approach of~\cite{kordaMCI}.
\item Being data-driven, the approach provides high-probability guarantees only, contrary to~\cite{kordaMCI} that provides guaranteed outer approximation to the MCI set. The sample complexity bounds are explicit, albeit depending on hard-to-estimate quantities.
\end{itemize}

Historically, the idea of using infinite-dimensional linear programming to address nonlinear optimal control problems
originated, to the best of our knowledge, with the work of Rubio~\cite{rubio1976extremal}, closely followed by the works of Vinter and Lewis~\cite{vinter1978equivalence,lewis1980relaxation}. The work of Rubio~\cite{rubio1976extremal} is in itself a follow-up on his earlier work~\cite{rubio1975generalized} that use the infinite-dimensional linear-programming embedding to study calculus of variations problems within the framework of generalized curves introduced by Young in~\cite{young1937generalized}. Without control, the fact that certain objects evolving along the flow of a nonlinear dynamical system obey linear relationships is classical; for example, the evolution of a density (or more generally a Borel measure) defined on the state space is governed by the so-called continuity equation whereas values of a function defined on the state-space evolve according to its adjoint, the transport equation.

Computationally, the infinite-dimensional LP approach had not been systematically exploited (to the best of our knowledge) until the 2000's, starting with the work~\cite{prajna2004nonlinear} based on the result of~\cite{rantzer2001dual} for stability analysis, albeit already the early work~\cite{rubio1976extremal} hints at a possible computational use and proposes a discretization-based algorithm. The first systematic use for optimal control in conjunction with semidefinite programming was in~\cite{lasserre2008nonlinear}. After that, this approach was successfully applied to a number of problems including the region of attraction~\cite{kordaROA,me_nolcos}, maximum invariant sets~\cite{kordaMCI}, invariant measures~\cite{korda2018invarMeasure,magron2018semidefinite}, bounding extreme values on attractors~\cite{goluskin2018bounding} or, very recently, nonlinear partial differential equations analysis and control~\cite{korda2018pdes,marx2018moment}. It should be mentioned that the occupation measure approach is also heavily employed in the stochastic processes literature, the survey of which is beyond the scope of this paper; see, e.g., \cite[Chapter~6]{HernandezLasserre}.

\paragraph{Notation}
The space of real-valued continuous functions on a compact set $\Xf \subset \Rb^n$ is denoted by $\C(\Xf)$ and we denote $\|f\|_{\C(\Xf)} = \sup_{x\in\Xf} |f(x)|$ the corresponding supremum norm.  The space of bounded Borel measurable functions on $\Xf$ is denoted by $\B(\Xf)$. Function composition is denoted by $\circ$, i.e., $(f\circ g)(x) = f(g(x))$. The Lipschitz constant of a function $f$ is denoted by $\mr{Lip}(f)$. \new{The symbols $\Nb$ and $\Nb_0$ denote the sets of all positive respectively nonnegative integers. The Euclidean and infinity norms on $\Rb^n$ are denoted by $\| \cdot \|_2$ and $\| \cdot\|_\infty$, respectively.}

\paragraph{Organization} We define the problem of invariant set computation in the uncontrolled setting in Section~\ref{sec:uncontProbDef}; Section~\ref{sec:infDimLP} presents an infinite-dimensional LP characterization of the MPI and discusses regularity of the optimizers to this LP. Section~\ref{sec:data} develops the data-driven approximation to this LP; in particular, Section~\ref{sec:proposedMethod} present succinctly the core of the proposed method. Section~\ref{sec:theory} provides theoretical analysis. Section~\ref{sec:probStatementCont} then extends the approach to the controlled setting (i.e., MCI set computation). Numerical examples are in Section~\ref{sec:NumEx} and longer proofs in Section~\ref{sec:proofs}.

\section{Problem statement (uncontrolled)}\label{sec:uncontProbDef}
For simplicity, we develop the theory in the uncontrolled setting and generalize to systems with control inputs in Section~\ref{sec:MCI}. Consider therefore the nonlinear discrete-time dynamical system
\begin{equation}\label{eq:sys}
x^+ = f(x),
\end{equation}
where $x\in \Rb^n$ is the current state and $x^+ \in \Rb^n$ is the successor state. The mapping $f:\Rb^n\to \Rb^n$ is assumed to be Borel measurable (but not necessarily Lipschitz continuous).

Given a compact constraint set
\[
\Xf \subset \Rb^n,
\]
the \emph{maximum positively invariant} (MPI) set associated to $f$ and $\Xf$ is defined by
\[
\Xf_\infty := \{x\in \Xf \mid f^{(k)}(x) \in \Xf \;\; \mr{for\ all}\;\; k\in \Nb_0 \},
\]
where $f^{(k)}$ denotes the $k$-times repeated application of $f$, i.e. 
\[
f^{(k)} = f\circ f\circ\ldots\circ f
\]
with $f^{(0)}$ being the identity map. In words, the  MPI set is the set of all initial conditions from $\Xf$ that remain in $\Xf$ forever under the action of the dynamical system~(\ref{eq:sys}).

In this work we do not assume the knowledge of $f$ but rather we have at our disposal the data in the form of $K$ pairs $(x_i,x_i^+)$, i.e.,
\begin{equation}\label{eq:data}
\mr{Data} = \big\{(x_i,x_i^+)\big\}_{i=1}^K
\end{equation}
with $x_i^+ = f(x_i)$. Note that no temporal ordering of the data pairs is assumed; in particular we \emph{do not require} that the data comes in the form of trajectories.

\paragraph{Goal} The main goal of the paper is to construct an approximation of the MPI set $\Xf_\infty$ from the sole knowledge of the data set~(\ref{eq:data}).

\begin{remark}[Output measurements]\label{rem:outMeas}
If only output rather than full state measurements are available, we can use the Taken's embedding theorem~\cite{takens1981detecting} and work with several consecutive output measurements instead of the single state measurement.
\end{remark}

\section{Infinite-dimensional LP formulation}\label{sec:infDimLP}
In this section we reformulate the MPI set problem as an infinite-dimensional linear programming problem (LP) in the space of bounded Borel measurable functions. A similar reformulation was already given in~\cite{kordaMCI}. However, the formulation in~\cite{kordaMCI} suffered from the subtle fact that the optimum was not attained in the space of continuous functions, which hampered theoretical analysis as well as lead to numerical issues connected with approximating discontinuous functions with polynomials. As we will show, the novel formulation presented here enjoys the property of the optimizer being \emph{attained} in the space of \emph{continuous functions}, provided that the transition mapping $f$ is continuous and $\Xf$ is sufficiently ``nice'', for example convex. 

The strategy to derive this formulation is to construct an auxiliary system $x^+= \bar f(x)$ and a ``stage cost'' $\bar l(x)$ such that
\begin{itemize}
\item  $\Xf$ is positively invariant under $\bar f$, i.e., $\bar f(\Xf) \subset \Xf$,
\item \new{$\bar f(x)$ equals $f(x)$ on the set $\{x\in\Xf \mid f(x) \in \Xf\}$},
\item $\bar l(x) = 0$ if $f(x) \in \Xf$ and $0 < \bar l(x) \le   1$ if $f(x) \notin \Xf$,
\item \new{For each $x\in \Xf$, the functions $\bar f(x)$ and $\bar l(x)$ can be easily computed from the knowledge of $x$, $f(x)$ and $\Xf$.}
\end{itemize}

The value function of a discounted optimal control problem (here without control) with this artificial dynamics and stage cost will then be nonnegative, bounded on $\Xf$ and its zero level set will be equal to $\Xf_\infty$ \new{(see the discussion after Eq.~(\ref{opt:opc}) below)}. The requirement that $0 < \bar l(x) \le 1$ when $f(x)\notin \Xf$ is just for mathematical convenience; any bounded cost function strictly positive whenever $f(x)\notin \Xf$ suffices.

In order to find such $\bar f$ and $\bar l$, let
\begin{equation}\label{eq:projDef}
\proj_\Xf(x) := \argmin_{y\in \Xf} \|x - y\|_2
\end{equation}
denote the Euclidean projection\footnote{If $\Xf$ is not convex,  the projection may not be unique. In this case, we assume that a measurable selection function is applied to the set of minimizers. Such measurable selection exists since the set of minimizers is compact~\cite[Theorem~18.19]{guide2006infinite}.} of $x$ onto $\Xf$ and let
\begin{equation}\label{eq:distDef}
\dist_\Xf(x) := \min\Big\{ \min_{y\in \Xf} \|x - y \|_2 ,\,1 \Big \}
\end{equation}
denote the Euclidean distance of $x$ from $\Xf$ saturated at one. One possible choice for $\bar f$ and $\bar l$ is then

\begin{equation}\label{eq:proj_dist}
\bar{f} = \proj_\Xf \circ f \qquad \mr{and}\qquad \bar{l} = \dist_\Xf\circ f.
\end{equation}

\looseness-1
\begin{remark}[Projection and distance functions]
The choice of the Euclidean projection and distance functions is not the only one possible. The projection function can be replaced by any measurable function that maps $f(\Xf)\setminus \Xf$ to $\Xf$ and is equal to the identity on $\Xf$ (in topology, such functions are referred to as retractions). The distance function can be replaced by any measurable function which is strictly positive on $f(\Xf) \setminus \Xf$ and equal to zero on $\Xf$. Convexity of $\Xf$ ensures additional regularity of these functions: it implies that the projection mapping is uniquely defined and Lipschitz continuous and that the distance function is Lipschitz continuous. \new{A less strict condition for the existence of Lipschitz functions satisfying the above-mentioned requirements is the existence of a bi-Lipschitz homeomorphism of $\Rb^n$ that maps $\Xf$ to a convex set (e.g., the unit ball)}, or the positive reach condition~\cite{federer1959curvature}.
\end{remark}

\new{Now we derive the infinite-dimensional LP that characterizes the MPI set; the claims made during the derivation will be justified rigorously in Theorem~\ref{thm:mainLP}.  See, e.g., \cite[Section 6]{HernandezLasserre} for a reference on the infinite-dimensional LP approach to optimal control and dynamic programming.

Consider the auxiliary ``optimal control problem'' (here without control) parametrized by the initial condition $x \in \Xf$
\begin{equation}\label{opt:opc}
\begin{array}{rclll}
v^\star(x) & = & \inf\; \sum_{k=0}^\infty \alpha^k \bar l(x_k) \vspace{1.5mm} \\
& \mathrm{s.t.} & \bar x_{k+1} = \bar f(x_k) \\
&& x_0 = x,
\end{array}
\end{equation}
where $\alpha \in (0,1)$ is a given discount factor and $v^\star:\Xf\to\Rb$ is the value function. We note that since there are no decision variables in~(\ref{opt:opc}), the ``inf'' is superfluous and the optimal value of~(\ref{opt:opc}) is simply the discounted sum of $\bar l$ along the trajectory of $x^+ = \bar f(x)$. Therefore, using the definition of $\bar f$ and $\bar l$ and the fact that $\Xf$ is invariant under $\bar f$, we have $0\le v^\star \le (1-\alpha)^{-1}$ on $\Xf$ and
\[
\Xf_\infty = \{x \in \Xf \mid v^\star(x) = 0\}. 
\]
Next, a direct computation shows that any $v:\Xf\to \Rb$ satisfying the Bellman inequality
\begin{equation}\label{eq:bell_ineq_aux}
v \le \bar l + \alpha\, v \circ \bar f \;\;\; \mr{on}\;\;\; \Xf
\end{equation}
satisfies $v \le v^\star$ on $\Xf$. We note that~(\ref{eq:bell_ineq_aux}) is \emph{linear} in $v$. It is therefore natural to try to maximize $v$ subject to the constraint~(\ref{eq:bell_ineq_aux}) in order to obtain $v^\star$. The objective function to maximize is \[
\int_\Xf v(x) \, d\lambda(x),
\]
where $\lambda \in \Mc(\Xf)_+$ is a given probability measure on $\Xf$. The support of $\lambda$ is required to be the entire constraint set $\Xf$ which is satisfied, for example, by the uniform probability measure on $\Xf$. Since this objective function is linear in $v$, we arrive at the following infinite-dimensional LP:
}
\begin{equation}\label{opt:LPinf}
\begin{array}{rclll}
d^* & = & \sup\limits_{v \in \Bc(\Xf)} & \displaystyle\int_{\Xf} v(x)\, d\lambda(x) \\
&& \mathrm{s.t.} &v(x)\le  \dist_\Xf(f(x)) + \alpha v(\proj_\Xf (f(x)))  \:\: &\forall\, x \in \Xf, \\
\end{array}
\end{equation}
\new{where the constraint of the LP is just~(\ref{eq:bell_ineq_aux}) written out explicitly.}


\new{The following result summarizes and expands on the previous discussion and will play a crucial role in the subsequent developments:}
\begin{theorem}\label{thm:mainLP}
For any Borel measurable transition mapping $f$, the supremum in~(\ref{opt:LPinf}) is attained by the bounded measurable function
\begin{equation}\label{eq:vstar_characterization}
v^\star(x) = \sum_{k=0}^\infty\alpha^k \bar l \big(\bar{f}^{(k)}(x)\big).      
\end{equation}
In addition:
\begin{enumerate}
\item We have $v^\star  = 0 $ on $\Xf_\infty$ and $v^\star(x) > 0$ on $\Xf \setminus \Xf_\infty$.
\item We have $\Xf_\infty \subset \{x \in \Xf \mid v(x) \le 0\}$ for any $v$ feasible in~(\ref{opt:LPinf}) and \[
\Xf_\infty = \{x\in \Xf \mid v^\star(x) = 0\} .
\]
\item If $f$ is continuous on $\Xf$ and $\mr{proj}_{\Xf}$ and $\mr{dist}_{\Xf}$ are continuous on $\Xf\cup f(\Xf)$, then $v^\star$ is uniformly continuous on $\Xf$. In particular, $v^\star$ is uniformly continuous on $\Xf$ if $f$ is continuous on $\Xf$ and $\Xf$ is convex.
\item If $f$ is Lipschitz continuous on $\Xf$ with Lipschitz constant $L_f$, $\Xf$ is convex, and $\alpha < L_f^{-1}$, then $v^\star$ is Lipschitz continuous on $\Xf$ with Lipschitz constant $1 / (1 - \alpha L_f)$.  
\end{enumerate}
\end{theorem}
\begin{proof}
See Section~\ref{sec:proofs}.
\end{proof}

\paragraph{Inner approximation} A linear program whose feasible solutions provide inner approximations to the MPI set can be obtained by characterizing the \emph{complement} of the MPI set. This idea was first used for the related problem of region of attraction estimation in~\cite{me_nolcos} and later for the MPI set in~\cite{oustry2019inner}.

\section{Data-driven approach}\label{sec:data}
In this section we describe how the infinite-dimensional LP can be approximated using the available data~(\ref{eq:data})
\[
\mr{Data} = \big\{(x_i,x_i^+)\big\}_{i=1}^K.
\]
Since the space of decision variables is infinite-dimensional and the number of constraints of~(\ref{opt:LPinf}) is infinite, they both have to be approximated.

\subsection{Finite-dimensional decision variable}
For the decision variable $v$ of~(\ref{opt:LPinf}), we choose a set of Lipschitz continuous basis functions
\[
\bs \beta(x) = [\beta_1(x),\ldots, \beta_N(x)]^\top
\]
and optimize over functions $v$ belonging to their span
\[
\mathcal{V}_N = \mr{span}\{\beta_1,\ldots,\beta_N\}.
\]

 This is a classical approach in optimal control, partial differential equations and many other fields \new{(see, e.g., \cite{vanRoy2003,ciarlet2002finite})}. This leads to the following optimization problem

\begin{equation}\label{opt:LPinf_N}
\begin{array}{rclll}
d_N & = & \sup\limits_{v \in\mathcal{V}_N} & \displaystyle\int_{\Xf} v(x)\, d\lambda(x) \\
&& \mathrm{s.t.} &v(x)\le  \dist_\Xf(f(x)) + \alpha v(\proj_\Xf (f(x)))  \:\: &\forall\, x \in \Xf. \\
\end{array}
\end{equation}

The following result is classical in the dynamic programming literature \new{(e.g.,\cite[Theorem~2]{vanRoy2003})}.
\begin{lemma}\label{lem:dp}
For every $N$ there exists a $\bar v_N$ feasible in~(\ref{opt:LPinf_N}) such that
\begin{equation}\label{eq:DP_eq1}
\int_\Xf  |v^\star - \bar v_N|\,d\lambda(x) \le  \|v^\star - \bar v_N\|_{\C(\Xf)} \le \frac{2}{1-\alpha} \min_{v\in\mathcal{V}_N}\|v^\star - v\|_{\C(\Xf)}.
\end{equation}
In particular, any optimal solution $v_N$ of (\ref{opt:LPinf_N}) satisfies
\begin{equation}\label{eq:DP_eq2}
\int_\Xf  |v^\star - v_N|\,d\lambda(x) \le  \frac{2}{1-\alpha} \min_{v\in\mathcal{V}_N}\|v^\star - v\|_{\C(\Xf)}
\end{equation}
\end{lemma}
\begin{proof}
See Section~\ref{sec:proofs}.
\end{proof}

Denote
\[
\Xf_N = \{x\mid v_N(x) \le 0\}
\]
and
\begin{equation}\label{eq:gvstar_def}
g_{v^\star}(\gamma) = \refMeas\big(\{x\mid v^\star(x) \in (0,\gamma]\}\big)
\end{equation}
for $\gamma > 0$. The function $g_{v^\star}$, capturing the rate of growth of $v^\star$ close to the boundary of $\Xf_\infty$, will play a crucial role in analysis of convergence of the finite-dimensional approximations in Theorem~\ref{thm:aux}. The function is illustrated in Figure~\ref{fig:gvstar}.
\begin{figure*}[h]
\begin{picture}(140,80)
\put(110,-5){\includegraphics[width=80mm]{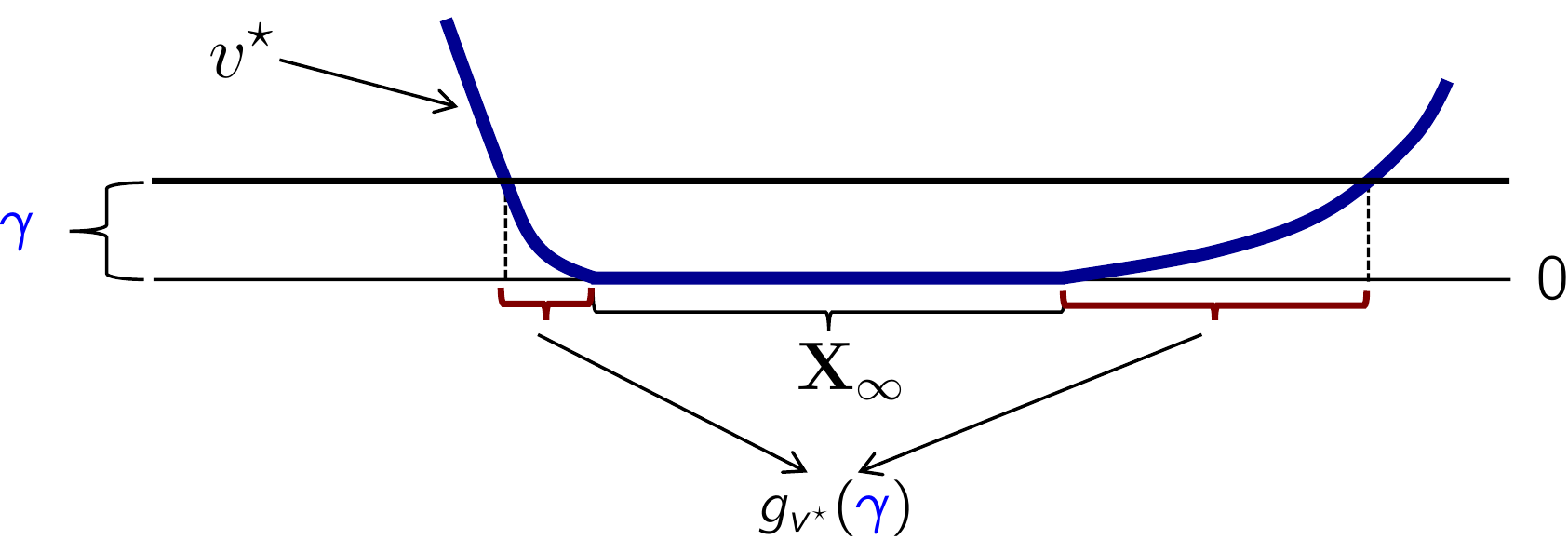}}
\end{picture}
\caption{\footnotesize Illustration of the function $g_{v^\star}$.}
\label{fig:gvstar}
\end{figure*}

\begin{lemma}
We have \[
\lim_{\gamma \to 0} g_{v^\star}(\gamma) = 0.
\]
\end{lemma}
\begin{proof}
This is a basic measure theoretic argument. Denoting $A_k = \{x\mid v^\star(x) \in (0,1/k]\}$, we have $A_{k+1} \subset A_k$, $\cap_{k=1}^\infty A_k = \emptyset$ and $g_{v^\star}(1/k) = \refMeas(A_k)$. Since $g_{v^\star}$ is a non-decreasing function of $\gamma$ we obtain
\[
\lim_{\gamma \to 0} g_{v^\star}(\gamma) = \lim_{k\to\infty } g_{v^\star}(1/k) = \lim_{k\to\infty }  \refMeas(A_k) = \refMeas(\cap_{k=1}^\infty A_k) = \lambda(\emptyset) = 0.
\]
\end{proof}

Define
\[
e_{\alpha,N} := \frac{2}{(1-\alpha)(1-\alpha L_f)}\min_{v\in \mathcal{V}_N}\|v^\star - v\|_{\C(\Xf)}.
\]
Notice that $e_{\alpha,N}$ goes to zero as $N\to\infty$ with the same rate as $\min_{v\in \mathcal{V}_N}\|v^\star - v\|_{\C(\Xf)}$. 

Then we have  the following result:

\begin{theorem}\label{thm:aux}
Let $v_N$ be a solution to~(\ref{opt:LPinf_N}). The following statements hold:
\begin{enumerate}
\item $\Xf_N \supset \Xf_{N+1}  \supset \ldots \supset \Xf_\infty$ for all $N \in \Nb$.
\item If $e_{\alpha,N} <1$ and $\bs\beta$ is a basis of the space of multivariate polynomials of degree at most $d$ (hence $N = \binom{n+d}{n}$) and $\alpha < 1/{L_f}$, then
\begin{equation}\label{eq:bound_vn}
\int_\Xf  |v^\star - v_N|\,d\lambda(x) \le \frac{2 C_{\Xf,n}}{(1-\alpha)(1-\alpha L_f)} \frac{1}{d},
\end{equation} 
where $C_{\Xf,n}$ is a constant depending only on the diameter of $\Xf$ and the dimension $n$.
\item Under the same assumptions, we have
\begin{equation}\label{eq:convRate_opt}
 \refMeas(\Xf_N\setminus \Xf_\infty) \le \inf_{\gamma > 0}\left\{\frac{2 C_{\Xf,n}}{(1-\alpha)(1-\alpha L_f)}\frac{1}{\gamma d}   + g_{v^\star}(\gamma)\right\}.
 \end{equation}
In particular
\begin{equation}\label{eq:convRate_sqrtd}
 \refMeas(\Xf_N\setminus \Xf_\infty) \le \frac{2 C_{\Xf,n}}{(1-\alpha)(1-\alpha L_f)}\frac{1}{\sqrt{d}}   + g_{v^\star}\left(\frac{1}{\sqrt{d}}\right)
 \end{equation}
\end{enumerate}
\end{theorem}
\begin{proof}
See Section~\ref{sec:proofs}.
\end{proof}

\begin{remark}[Optimizing the convergence rate]
The question of the optimal $\gamma$ in (\ref{eq:convRate_opt}) depends on the behavior of $g_{v^\star}$ close to zero. For example, assume that $g_{v^\star}(\gamma) = \mathcal{O}(\gamma^b)$ for $\gamma \to 0$ with $b > 0$. Then, the rate of convergence in~(\ref{eq:convRate_opt}) is optimized by $\gamma =\mathcal{O}\left(\frac{1}{d^{\frac{1}{b+1}}}\right)$, leading to the rate $\refMeas(\Xf_N\setminus \Xf_\infty) = \mathcal{O}\left(\frac{1}{d^{1-\frac{1}{b+1}}}\right)$. Notice that this rate tends to $\frac{1}{d}$ for $b\to \infty$, which is the optimal worst-case rate of approximating Lipschitz functions by polynomials.
\end{remark}

\begin{remark}[Why polynomials]\label{rem:whyPolynomials}\looseness-1
There is nothing special about using polynomials in Theorem~\ref{thm:aux}; analogous rates of convergence can be derived for any basis functions, given the knowledge of the approximation rate of Lipschitz continuous functions from within this basis.
\end{remark}

\subsection{Data-driven constraint approximation}\label{sec:sample_approx}
In this section, we describe how to proceed in a data-driven fashion, i.e., when $f$ is not known explicitly but only the data $\big\{(x_i,x_i^+)\big\}_{i=1}^K$ is available. The basic, simple, idea is to impose the constraint of the infinite-dimensional LP~(\ref{opt:LPinf}) only on the available data points while controlling the interplay between the number of data points $K$ and the number of basis functions $N$.

In addition to this, we impose the constraint
\begin{equation}\label{eq:boundConst}
-1 \le v(x) \le (1-\alpha)^{-1}
\end{equation}
on a second, \emph{artificial}, set of samples
\begin{equation}\label{eq:dataArtif}
\mr{Data}' = \{z_i\}_{i=1}^{K'}
\end{equation}
which we assume to be \emph{unisolvent} with respect to the basis $\bs\beta$:
\begin{assumption}\label{as:unisolvent}
The aritificial data set~(\ref{eq:dataArtif}) is unisolvent with respect to the basis $\bs\beta$, i.e.,
\[
\cf^\top\bs\beta(z_i) = 0 \; \forall i\in \{1,\ldots, K'\} \;\;\;\; \emph{if and only if}\;\;\;\; \cf = 0.
\]
\end{assumption}
This second artificial data set is fully in our control and therefore we can easily satisfy the unisolvency assumption, for example by drawing $z_i$ from the uniform distribution over $\Xf$ with $K' > N$ under mild conditions on $\bs\beta$ and $\Xf$ (e.g., $\bs\beta$ being polynomial and $\Xf$ having a non-empty interior). \new{This assumption implies that the mapping $v\mapsto \max_{i=1,\ldots,K'}|v(z_i)|$ defines a norm on $\Vc_N$; since $\Vc_N$ is finite dimensional this norm is equivalent to the coefficient norm given by $v\mapsto \|\cf\|_2$ for $v = \cf^\top \bs\beta$.}

We note that the constraint~(\ref{eq:boundConst}) is redundant for the infinite-dimensional LP~(\ref{opt:LPinf}) since $0\le v^\star\le (1-\alpha)^{-1}$ on $\Xf$. The reason to include this constraint in the data-driven approximation is in order to ensure boundedness of the feasible set of the finite-dimensional LP solved. This is important only if the original data set~(\ref{eq:data}) is not large enough (or not rich enough) in which case this additional constraint acts as a regularizer and prevents the finite-dimensional LP approximation to be unbounded.

\subsection{The proposed method}\label{sec:proposedMethod}

Now we formulate the finite-dimensional LP solved by our method, which lies at the heart of the proposed approach. This LP is nothing but the LP~(\ref{opt:LPinf_N}) with the constraint imposed only on the available data~(\ref{eq:data}) and with the additional constraint~(\ref{eq:boundConst}) imposed on the artificial data set~(\ref{eq:dataArtif}). We write down the LP with an explicit parametrization of the decision variable as $v = \bs\beta(x)^\top \cf$, with $\cf\in\Rb^N$.

The finite-dimensional LP solved by the proposed method reads:

\begin{equation}\label{opt:LPfinite}
\begin{array}{rclll}
d_{N,K} & = & \sup\limits_{\cf \in \Rb^N} & \displaystyle    \zf^\top \cf       \\
&& \mathrm{s.t.} & \bs\beta(x_i)^\top \cf  \le  \dist_\Xf(x_i^+) + \alpha \bs\beta(\proj_\Xf (x_i^+))^\top \cf  \:\: &\forall\, i\in 1,\ldots,K,\vspace{2mm} \\
&&& -1 \le \bs\beta(z_i)^\top \cf \le (1 - \alpha)^{-1}  \:\: &\forall\, i\in 1,\ldots,K',\ 
\end{array}
\end{equation}
where the vector $\zf$ is defined by\footnote{For simplicity of analysis we assume that the integral of each basis function over $\Xf$ can be evaluated analytically. If this is not the case, this integral could be approximated with high accuracy using Monte Carlo sampling techniques or numerical quadrature. We do not analyze the error due to such approximation here.}
\begin{equation}\label{eq:zf_def}
\zf := \int_{\Xf} \bs\beta(x)\,d\lambda(x) \in \Rb^N. 
\end{equation}

Let $\cf_{N,K}$ denote an optimal solution to~(\ref{opt:LPfinite}) and let \[
v_{N,K}(x) = \bs\beta(x)^\top \cf_{N,K}.
\]
The approximation of the MPI set, and the key object of the study of this paper, is then defined by
\begin{equation}\label{eq:XNK}
\Xf_{N,K} = \{x\in\Xf \mid v_{N,K}(x) \le 0 \}.
\end{equation}
\new{It is important to note that the constraint~(\ref{eq:bell_ineq_aux}) is imposed on only finitely many samples in the LP~(\ref{opt:LPfinite}) and therefore it is not guaranteed that the set $\Xf_{N,K}$ is an outer approximation of $\Xf_\infty$. This possible ``misclassification'' is studied in Section~\ref{sec:theory} in a probabilistic setting; Section~\ref{sec:guarantees} then describes how \emph{guaranteed} outer approximations can be obtained if additional modeling information is available.}

\paragraph{Discussion} Optimization problem~(\ref{opt:LPfinite}) is a finite-dimensional linear programming problem that can be readily solved using off-the-shelf software. To be be specific, observe that the vectors $\bs\beta(x_i)$ and $\bs\beta(z_i)$ are just fixed vectors in $\Rb^N$ and define
\[
\Af_1 = \begin{bmatrix} \bs\beta(x_1)^\top - \alpha\bs\beta(\proj_{\Xf}(x_1^+))^\top \\ \vdots \\  \bs\beta(x_K)^\top - \alpha\bs\beta(\proj_{\Xf}(x_K^+))^\top \end{bmatrix},\quad  \Af_2  =  \begin{bmatrix}\bs\beta(z_1)^\top \\ \vdots \\ \bs\beta(z_{K'})^\top \end{bmatrix}, \quad  \mathbf{b}_1 = \begin{bmatrix}
\dist_\Xf(x_1^+) \\ \vdots \\ \dist_\Xf(x_K^+)
\end{bmatrix}
\]
and
\[
\Af = \begin{bmatrix}
\Af_1 \\ \Af_2 \\ -\Af_2
\end{bmatrix},\quad \mathbf{b} = \begin{bmatrix} \mathbf{b}_1 \\ (1-\alpha)^{-1}\ones_{K'} \\ \ones_{K'} \end{bmatrix},
\]
where $\ones_{K'} \in \Rb^{K'}$ denotes the column vector of ones of length $K'$. Then, the LP~(\ref{opt:LPfinite}) is equivalent to
\begin{equation}\label{opt:LPfinite_dumbedDown}
\begin{array}{rclll}
d_{N,K} & = & \sup\limits_{\cf \in \Rb^N} & \displaystyle    \zf^\top \cf       \\
&& \mathrm{s.t.} & \bf A\cf  \le \mathbf{b}. 
\end{array}
\end{equation}

\new{\paragraph{Computational complexity} The worst-case computational complexity to solve~(\ref{opt:LPfinite_dumbedDown}) is polynomial in $(N,K)$. However, the theoretical analysis of Section~\ref{sec:theory} suggests that the number of basis functions $N$ and the number of samples $K$ required to achieve a given level of accuracy may grow exponentially with the state-space dimension. Although these are only upper bounds, we conjecture that an exponential dependence is unavoidable at this level of generality. It remains an open theoretical question to derive matching lower bounds as well as to determine whether the exponential dependence can be overcome for particular classes of systems (e.g., sparse systems or systems with symmetries).
}

\section{Theoretical analysis}\label{sec:theory}
In this section we study the convergence of the sets $\Xf_{N,K}$, with respect to both $N$ and $K$, to the MPI set~$\Xf_\infty$.  To avoid mathematical pathologies related to sampling, we make the following assumption on the constraint set $\Xf$:
\begin{assumption}\label{as:X}
The set $\Xf$ is compact and is equal to the closure of its interior.
\end{assumption}
This is a mild assumption satisfied in most applications. The results extend to the case where the dynamics evolves (or is required to evolve in the controlled case) on a lower-dimensional variety as long as one ensures that the set of samples $(x_i)_{i=1}^\infty$ is dense in $\Xf$ in the limit as $K$ tends to infinity with probability one. Under Assumption~\ref{as:X}, it suffices to sample uniformly over $\Xf$; in what follows we quantify the rate at which samples become dense in $\Xf$ under this assumption.

\begin{definition}[$\epsilon$ net]
A collection of points $\{x_i\}_{i=1}^K$ is an $\epsilon$ net for the set $\Xf$ if $\Xf \subset \bigcup_{i=1}^K B_\epsilon(x_i)$, where $B_\epsilon(x) = \{y \mid \|y - x\|_\infty \le \epsilon  \}$.
\end{definition}

The following Lemma is an elementary probabilistic probabilistic result:
\begin{lemma}\label{lem:epsNet}
Let assumption~\ref{as:X} hold, let the points $\{x_i\}_{i=1}^K$ be drawn independently from the uniform distribution over $\Xf$ and let $D$ be the diameter of $\Xf$ and let
\begin{equation}\label{eq:sampNum}
K \ge \frac{\log(\frac{1}{\delta})  + n\frac{\epsilon}{2D}}{\log\left(   \frac{1}{1-\frac{ \epsilon^n}{2^nD^n}  }  \right)}
\end{equation}
with $\delta \in (0,1]$. Then $\{x_i\}_{i=1}^K$ is an $\epsilon$ net for $\Xf$ with probability at least $1-\delta$.
\end{lemma}
\begin{proof}
See Section~\ref{sec:proofs}.
\end{proof}

We shall use the following natural assumption on the basis $\bs\beta$ and set $\Xf$:
\begin{assumption}\label{ass:basis2}
The basis $\bs\beta$ and the set $\Xf$ are such that
\[
\vol(\{x\in\Xf \mid \cf^\top \bs\beta(x) = 0\}) > 0 \qquad \mr{if\;and\;only\;if}\qquad \cf = 0.
\]
\end{assumption}
In simple terms, this assumption asks that the volume of the zero-level sets of the functions from $\Vc_N$ be zero, unless the function is identically zero on $\Xf$. Again, this assumption is satisfied, for example, if $\bs\beta$ is a polynomial basis and $\Xf$ has a non-empty interior.

Denote
\[
L_{n,N} = \sup \big\{ \mr{Lip}(v - \mr{dist}_{\Xf}\circ f - \alpha v \circ \proj_\Xf\circ f ) \; \mid \; v\  \mr{feasible\ in\ }(\ref{opt:LPfinite})    \big\}.
\]

The following result summarizes convergence properties of the proposed method.
\begin{theorem}\label{thm:mainBound}
Let Assumption~\ref{as:unisolvent} and \ref{as:X} hold, let the points $(x_i)_{i=1}^K$ in the data set~(\ref{eq:data}) be drawn from the uniform distribution over $\Xf$. Then we have
\begin{enumerate}
\item If also Assumption~\ref{ass:basis2} holds and $f$ is continuous, then for any fixed $N$, with probability one
\[
\lim_{K \to \infty} \vol( \Xf_\infty\setminus \Xf_{N,K} ) = 0.
\]
\item Let $f$ be Lispchitz continuous with Lipschitz constant $L_f$ and let $\bs\beta$ be a basis of the space of multivariate polynomials of degree at most $d$ (hence $N = \binom{n+d}{n}$) and $\alpha < 1/{L_f}$. Given accuracy $\epsilon > (0,1)$, confidence $\delta \in (0,1)$ and
\begin{equation}\label{eq:K_lower_bound}
K \ge \frac{\log\,\delta - n\log(\frac{1}{\zeta})}{\log\big(1- \zeta^n \big)}  = \frac{\log(\frac{1}{\delta}) + n\log(\frac{1}{\zeta})  }{ \log\big( \frac{1}{1-\zeta^n}  \big) },
\end{equation}
where 
\[
\zeta= \frac{\epsilon(1-\alpha)}{2 D L_{n,N}},
\]
we have with probability at least $1-\delta$ that
\begin{equation}\label{eq:mainBound}
\refMeas(\Xf_{N,K} \setminus \Xf_\infty) \le \frac{4C_{\Xf,n}}{(1-\alpha)(1-\alpha L_f)}\frac{1}{\sqrt{d} + \epsilon d} + \frac{\epsilon}{\frac{1}{\sqrt{d}}  +\epsilon } + g\left(\frac{1}{\sqrt{d}}\right)
\end{equation}
for all $ d\ge 2C_{\Xf,n}[(1-\alpha)(1-\alpha L_f)]^{-1}$.
\end{enumerate}

\end{theorem}
\begin{proof}
See Section~\ref{sec:proofs}.
\end{proof}

\paragraph{Discussion of Theorem~\ref{thm:mainBound}}  The bound of Theorem~\ref{thm:mainBound}  can be further simplified using the fact that $\log x \ge 1 - \frac{1}{x}$; we have
\[
\frac{\log(\frac{1}{\delta}) + n\log(\frac{1}{\zeta})  }{ \log\big(\frac{1}{1-\zeta^n} \big) } \le \frac{\log(\frac{1}{\delta}) + n\log(\frac{1}{\zeta})  }{\zeta^n}
\]
so a simplified, less-precise, sample bound is
\[
K \ge \frac{\log(\frac{1}{\delta}) + d\log(\frac{1}{\zeta})  }{\zeta^n} = \tilde{\mathcal{O}}(\epsilon^{-n}),
\]
where  $\tilde{\mathcal{O}}$ signifies the classical big-O notation modulo logarithmic terms in $\epsilon$. Using the fact that $\log \frac{1}{\epsilon} < \frac{1}{\epsilon}$ and inverting the expression for $K$, we get
\[
\epsilon = \mathcal{O}(K^{\frac{1}{n+1}}).
\]
Setting $d = K^{\frac{1}{n+1}}$, we obtain the bound
\[
\lambda(\Xf_{N,K}\setminus \Xf_\infty) \le \mathcal{O}(K^{-\frac{1}{2(n+1)}}) + g_{v^\star}(K^{-\frac{1}{2(n+1)}}),
\]
which goes to zero as the number samples $K$ tends to infinity. Finally, we note that Remark~\ref{rem:whyPolynomials} applies to Theorem~\ref{thm:mainBound} as well.

\subsection{Guaranteed approximations}\label{sec:guarantees}
It is a natural questions to ask whether one can, possibly with some additional knowledge, obtain a guaranteed outer approximation of $\Xf_\infty$. The following lemma and its immediate corollary go in this direction.
\begin{lemma}\label{eq:lemmaGuarant}
Any solution $v_{N,K}$ to~(\ref{opt:LPfinite}) satisfies
\begin{equation}\label{eq:approxGuarant}
\Xf_{N,K}^{\mr{G}} := \big\{ x\in\Xf\mid v_{N,K}(x) \le (1-\alpha)^{-1}\sup_{z\in \Xf}E(z)  \big\} \supset \Xf_\infty,
\end{equation}
where the function $E:\Xf\to\Rb$ is defined by
\[
E(x) := v_{N,K}(x) - \dist_\Xf(f(x)) - \alpha v_{N,K} (\proj_\Xf(f(x))).
\]
\end{lemma}
\begin{proof}
By definition of $E(x)$, which is just the slack in the constraint of~(\ref{opt:LPinf_N}), it follows that
\[
v_{N,K} - (\bar l + E )- \alpha v_{N,K}\circ \bar f = 0,
\]
which is a Bellman equality for the dynamics $\bar f$, stage cost $\bar l + E$ and discount factor $\alpha$. Therefore, using the same computation as in the proof of Theorem~\ref{thm:mainLP}, we obtain 
\[
v_{N,K}  = \sum_{k=0}^\infty\alpha^k (\bar l + E )\circ \bar f^{(k)}.
\]
Since $\bar l(x) = 0$ on $\Xf_\infty$, it follows that $\bar l \circ \bar f^{(k)} = 0$ on the MPI set $\Xf_\infty$. Therefore, for all $x\in\Xf_\infty$ it holds 
\[
v_{N,K}(x)  = \sum_{k=0}^\infty\alpha^k E \circ \bar f^{(k)}(x) \le \sup_{z\in \Xf}E(z) \sum_{k=0}^\infty\alpha^k \le  (1-\alpha)^{-1} \sup_{z\in \Xf}E(z),
\]
which implies that $\Xf_{N,K}^{\mr{G}} \supset \Xf_\infty$ as desired.
\end{proof}
\begin{corollary}\label{cor:guarant} Let $v_{N,K}$ denote a solution to (\ref{opt:LPfinite}) and let $\epsilon$ denote the diameter of the smallest $\epsilon$ net for $\Xf$ with the centers $(x_i)_{i=1}^K$, i.e.,
\[
\epsilon = \inf_{\delta > 0} \big\{\delta \mid \cup_{i=1}^K B_{\delta}(x_i)\supset \Xf \big\}.
\]
Then, if $f$ is Lipschitz continuous with a Lipschitz constant $L_f$, we have
\[
\tilde \Xf_{N,K}^{\mr{G}} := \big\{ x\in\Xf\mid v_{N,K}(x) \le (1-\alpha)^{-1}  \epsilon[\mr{Lip}(v_{N,K})(1 + \alpha L_f) + L_f]    \big\} \supset \Xf_\infty.
\]
\end{corollary}
\begin{proof}
This follows from the fact that
\[
E(x) \le \epsilon[ \mr{Lip}(v_{N,K}) + L_f + \mr{Lip}(v_{N,K})L_f   ]
\]
for all $x\in \Xf$, which is a consequences of the definition of $E$ and $\epsilon$.
\end{proof}
\paragraph{Discussion} Both Lemma~\ref{eq:lemmaGuarant} and Corollary~\ref{cor:guarant} require information which cannot be extracted from the data samples $(x_i,x_i^+)_{i=1}^{K}$ unless further modeling assumptions are made.  For Corollary~\ref{cor:guarant}, the only modeling assumption is the knowledge of an upper bound on the Lipschitz constant of $f$.

\paragraph{Conservative outer approximations} If no further information is available, one can resort to ex-post validation techniques, providing more conservative outer approximations obtained by estimating the $\sup_{z\in \Xf}E(z)$ in~(\ref{eq:approxGuarant}). Specifically, we split the data set $\mr{Data} = (x_i,x_i^+)_{i=1}^K$ in two disjoint sets of size $K_1$ and $K_2$ and solve the LP~(\ref{opt:LPfinite}) using only the first data set. Then we use the second data set to estimate $\sup_{z\in \Xf}E(z)$, i.e., we compute
\begin{equation*}
\bar E = \max_{i=1,\ldots, K_2} [v_{N,K}(x_i) - \dist_\Xf(x_i^+) - \alpha v_{N,K} (\proj_\Xf(x_i^+))].
\end{equation*}
The conservative approximation is then defined by
\begin{equation}\label{eq:conservApprox}
\Xf_{N,K}^{\mr{C}} := \big\{ x\in\Xf\mid v_{N,K}(x) \le (1-\alpha)^{-1}\bar E \big\}.
\end{equation}
Without further assumptions,  this approximation is not guaranteed to be an outer approximation and one can hope for probabilistic guarantees only, akin to those obtained in Theorem~\ref{thm:mainBound}. However, in all numerical examples tested, we observed that $\Xf_{N,K}^{\mr{C}}$ provided an outer approximation when the data set was split in two equal parts. A rigorous analysis of this conservative approximation is left for future work.

\section{Problem statement (controlled)}\label{sec:probStatementCont}
Now we briefly describe how the framework extends to the problem of the maximum controlled invariant set computation. We will use the same notation as in the uncontrolled setting. Consider the discrete time controlled system
\begin{equation}\label{eq:sys_cont}
x^+ = f(x,u)
\end{equation}
with $x \in \Rb^n$, $x^+ \in \Rb^n$ being the current respectively successor state and $u \in \Rb^m$ the control. Given compact state and control constraint sets
\[
\Xf\subset \Rb^n\,,\qquad  \Uf \subset \Rb^m,
\]
the \emph{maximum controlled invariant} (MCI) set
\begin{equation}
\Xf_\infty =\{ x_0 \in \Rb^n \mid \exists\, (u_k)_{k=0}^\infty \;\; \mr{s.t.} \;\; x_{k+1} = f(x_k,u_k),\; x_k \in \Xf,\; u_k\in\Uf   \}.
\end{equation}
In words, the MCI set is the set of all initial states of~(\ref{eq:sys_cont}) that can be kept inside the constraint set $\Xf$ forever using admissible control inputs.

\subsection{Infinite-dimensional LP characterization of MCI set}
The MCI set is characterized by the following LP analogous to~(\ref{opt:LPfinite})
\begin{equation}\label{opt:LPinf_cont}
\begin{array}{rclll}
d^* & = & \sup\limits_{v \in \Bc(\Xf)} & \displaystyle\int_{\Xf} v(x)\, d\lambda(x) \\
&& \mathrm{s.t.} &v(x)\le  \dist_\Xf(f(x,u)) + \alpha v(\proj_\Xf (f(x,u)))  \:\: &\forall\, (x,u) \in \Xf \times \Uf, \\
\end{array}
\end{equation}
where $\alpha \in (0,1)$.

As in~(\ref{eq:proj_dist}) define $\bar f:\Xf\times \Uf \to \Xf$ and $\bar l :\Xf\times \Uf \to [0,1]$ by
\begin{equation}\label{eq:proj_dist_cont}
\bar{f} = \proj_\Xf \circ f \qquad \mr{and}\qquad \bar{l} = \dist_\Xf\circ f,
\end{equation}
where $\proj_\Xf$ and $\dist_\Xf$ are defined in~(\ref{eq:projDef}) and (\ref{eq:distDef}). We have the following theorem:
\begin{theorem}\label{thm:mainLP_cont}
For any Borel measurable transition mapping $f:
\Rb^n\times\Rb^m\to\Rb^n$, the supremum in~(\ref{opt:LPinf_cont}) is attained by the bounded measurable function
\begin{equation}\label{eq:vstar_characterization_cont}
v^\star(x) = \inf\left\{\sum_{k=0}^\infty\alpha^k  \bar l(x_k,u_k) \mid x_{k+1} = \bar f(x_k,u_k),\, x_0 = x, \, u_k \in \Uf \right\}.      
\end{equation}
In addition:
\begin{enumerate}
\item We have $v^\star  = 0 $ on $\Xf_\infty$ and $v^\star(x) > 0$ on $\Xf \setminus \Xf_\infty$.
\item We have $\Xf_\infty \subset \{x \in \Xf \mid v(x) \le 0\}$ for any $v$ feasible in~(\ref{opt:LPinf}) and \[
\Xf_\infty = \{x\in \Xf \mid v^\star(x) = 0\} .
\]
\item If $f$ is jointly continuous on $\Xf \times \Uf$ and $\mr{proj}_{\Xf}$ and $\mr{dist}_{\Xf}$ are continuous on $\Xf\cup f(\Xf\times \Uf)$, then $v^\star$ is uniformly continuous on $\Xf$. In particular, $v^\star$ is uniformly continuous on $\Xf$ if $f$ is continuous on $\Xf \times \Uf$ and $\Xf$ is convex.
\item If $f$ is jointly Lipschitz continuous on $\Xf \times \Uf$ with Lipschitz constant $L_f$, $\Xf$ is convex, and $\alpha < L_f^{-1}$, then $v^\star$ is Lipschitz continuous on $\Xf$ with Lipschitz constaint $1 / (1 - \alpha L_f)$.  
\end{enumerate}
\end{theorem}
\begin{proof}
See Section~\ref{sec:proofs}.
\end{proof}

\subsection{Data-driven approach for MCI set}\label{sec:MCI}
In this section we describe how the infinite-dimensional LP~(\ref{opt:LPinf_cont}) can be approximated using the available data
\begin{equation}\label{eq:data_cont}
\mr{Data} = \big\{(x_i,x_i^+)\big\}_{i=1}^K,
\end{equation}
where $x_i^+ =  f(x_i,u_i)$. Notice that, very interestingly, we \emph{do not require} the knowledge of the control inputs $u_i$ that effected the transitions $x_i \to x_i^+$.

As in the case without control, first we restrict the space of decision variables to a finite dimensional subspace spanned by the Lipschitz continuous basis functions
\[
\bs \beta(x) = [\beta_1(x),\ldots, \beta_N(x)]^\top
\]
and we optimize over functions $v$ belonging to their span
\[
\mathcal{V}_N = \mr{span}\{\beta_1,\ldots,\beta_N\}.
\]
This leads to
\begin{equation}\label{opt:LPinf_N_cont}
\begin{array}{rclll}
d_N & = & \sup\limits_{v \in\mathcal{V}_N} & \displaystyle\int_{\Xf} v(x)\, d\lambda(x) \\
&& \mathrm{s.t.} &v(x)\le  \dist_\Xf(f(x,u)) + \alpha v(\proj_\Xf (f(x,u)))  \:\: &\forall\, (x,u) \in \Xf\times \Uf. \\
\end{array}
\end{equation}

The data-driven approximation to~(\ref{opt:LPinf_N_cont}) reads
\begin{equation}\label{opt:LPfinite_cont}
\begin{array}{rclll}
d_{N,K} & = & \sup\limits_{\cf \in \Rb^N} & \displaystyle    \zf^\top \cf       \\
&& \mathrm{s.t.} & \bs\beta(x_i)^\top \cf  \le  \dist_\Xf(x_i^+) + \alpha \bs\beta(\proj_\Xf (x_i^+))^\top \cf  \:\: &\forall\, i\in 1,\ldots,K,\vspace{2mm} \\
&&& -1 \le \bs\beta(z_i)^\top \cf \le (1 - \alpha)^{-1}  \:\: &\forall\, i\in 1,\ldots,K',\ 
\end{array}
\end{equation}
where the samples $z_i$ come from a second, artificial, data set
\[
\mr{Data}' = \{z_i\}_{i=1}^{K'}
\]
unisolvent with respect to the selected basis functions, as in Section~\ref{sec:sample_approx}; the vector $\bf z$ is defined by~(\ref{eq:zf_def}).

\paragraph{Theoretical analysis for MCI set} Since all the results for the MPI set rely only on the regularity of $v^\star$ which is, by Theorem~\ref{thm:mainLP_cont}, the same as for the MCI set, they extend immediately to the case of the MCI set with the proofs being verbatim copies. The sampling bounds, notably~(\ref{eq:K_lower_bound}), need to be adjusted for the dimension of $\Uf$, i.e., $n$ must be replaced by $n+m$ \new{and the underlying assumption is that the state and control samples are drawn from the uniform distributions over $\Xf$ and $\Uf$, respectively}.

\section{Numerical examples}\label{sec:NumEx}
This section presents several numerical examples to demonstrate the approach. All problems were coded in Matlab with the help of Yalmip~\cite{yalmip}. Linear programs were solved using Gurobi; semidefinite programs using Mosek 8. The performance of the methods compared is assessed in terms of the volume error of the approximation defined as
\begin{equation}\label{eq:volerr}
\mr{Volume \ error} = 100\cdot \frac{\vol(\Xf_{\mr{approx}} \setminus \Xf_\infty) }{\vol(\Xf_\infty)} \; [\%],
\end{equation}
where $\Xf^\infty$ is the true MPI or MCI set and $\Xf_{\mr{approx}}$ is a candidate outer approximation. We also report the ``misclassification'' due to finitely many data samples. The misclassification is defined as
\begin{equation}\label{eq:misclass}
\mr{Misclassification} = 100\cdot \frac{\vol(   \Xf_{\mr{approx}}^c \cap \Xf_\infty )  }{\vol(\Xf_\infty)} \; [\%],
\end{equation}
which is the proportion of points classified as outside of the MPI or MCI set whereas they are in reality inside (note that all the algorithms compared aim at obtaining outer approximations).

\paragraph{Bases used} As the basis, we will utilize either the monomial basis  (i.e., functions of the form $\Pi_{i=1}^n x_i^{\gamma_i}$ with $\gamma\in\Nb^n$) up to a total degree $d$ (i.e., $\sum_i \gamma_i \le d$); for a degree $d$, the number of basis functions created in this way is $N = \binom{n+d}{n}$. We will also use thin-plate-spline radial basis functions (RBFs); given $N$ centers $c_1,\ldots,c_N$, this basis is comprised of the functions
\[
x\mapsto \|x - c_i\|^2_2\log\|x - c_i\|_2,\quad i \in \{1,\ldots,N\}.
\]
The subspace spanned by these basis functions depends on the choice of the centers; in all the examples considered we generated  the centers randomly with a uniform distribution over the constraint set $\Xf$. Alternatively, one could adapt the centers to the data at hand, e.g., cluster the data to $N$ clusters and choose the $c_i$'s to be the cluster centroids.

Matlab code for the examples is available at

\begin{center}
\url{https://homepages.laas.fr/mkorda/MCI_data_driven.zip}
\end{center}

\subsection{Julia map}
First we consider the so-called Julia map, which is a recurrence of the form
\begin{equation}\label{eq:julia}
x^+ = \begin{bmatrix}
x_1^2 - x_2^2 + a_1 \\ 2x_1x_2 + a_2
\end{bmatrix}
\end{equation}
with $a = (-0.7, 0.2)$; the value of $a$ influences the shape of the MPI set -- see~\cite{kordaMCI} for experiments with different values of $a$. The state constraint is the unit ball $\Xf = \{x\in \Rb^2\mid \|x\|_2 \le 1\}$. 

\paragraph{Comparison with SDP} First, we compare the proposed algorithm with the SDP based approach of~\cite{kordaMCI}. In order to do so, we set $\bs\beta_N$ to be the monomial basis of total degree $d$ and use $3\cdot 10^4$ data points sampled uniformly over $\Xf$ (the effect of decreasing the sample size is investigated later). The discount factor $\alpha$ is set to 0.6. Figure~\ref{fig:JuliaMonom} shows the results of the proposed data-driven approach with $d= 10$ and $d = 18$. Table~\ref{tab:juliaDatavsSDP} reports the volume error and misclassification in comparison with the SDP based approach of~\cite{kordaMCI}; we notice that whenever neither approach encounters numerical problems, the results are very similar. However, due to the ill-condioning of the monomial basis, the SDP-based approach~\cite{kordaMCI} encounters numerical problems beyond degree 14. Beyond degree 20, also the LP~(\ref{opt:LPfinite}) in the data-driven method become too ill-conditioned\footnote{A more careful implementation (e.g., with a different polynomial basis used or a problem-specific preconditioning) may improve the performance of both methods. The ill-conditioning is a subtle issue and the numerical results presented here are not be taken as representative of the relative performance of the two methods.} to be accurately solved by Gurobi.

\paragraph{RBFs} Next, in Table~\ref{tab:RBFs}, we report the volume error and misclassification for the data-driven approach with the thin-plate spline RBFs; this basis cannot be easily used with the SDP-based approaches due to the lack of efficient nonnegativity certificates for this basis. We observe that for a given number of basis functions, the monomial basis provides a tighter approximation of the MPI set; however, the RBF basis is much better conditioned and allows the LP~(\ref{opt:LPfinite}) to be solved for much larger values of $N$, thereby achieving higher accuracy of the approximation. Figure~\ref{fig:JuliaRBF} shows the RBF approximations with $N= 200$ and $N = 600$. We also investigated the effect of increasing the size of the basis even further in order to observe non-negligible misclassification due to ``overfitting''. This occurs for $N=1000$. Luckily, the misclassification can be eliminated by using the conservative approximation~(\ref{eq:conservApprox}) (with the data set split in two equal parts). Figure~\ref{fig:misclass} depicts both the non-conservative and conservative approximations, with the misclassified points depicted in orange.

\paragraph{Small data limit} Next, in Figure~\ref{fig:JuliaLowSamp}, we investigate performance with a low number of data samples, namely 200 and 1000. We depict both the samples themselves as well as the approximations.

\begin{figure*}[h]
\begin{picture}(140,170)
\put(50,0){\includegraphics[width=60mm]{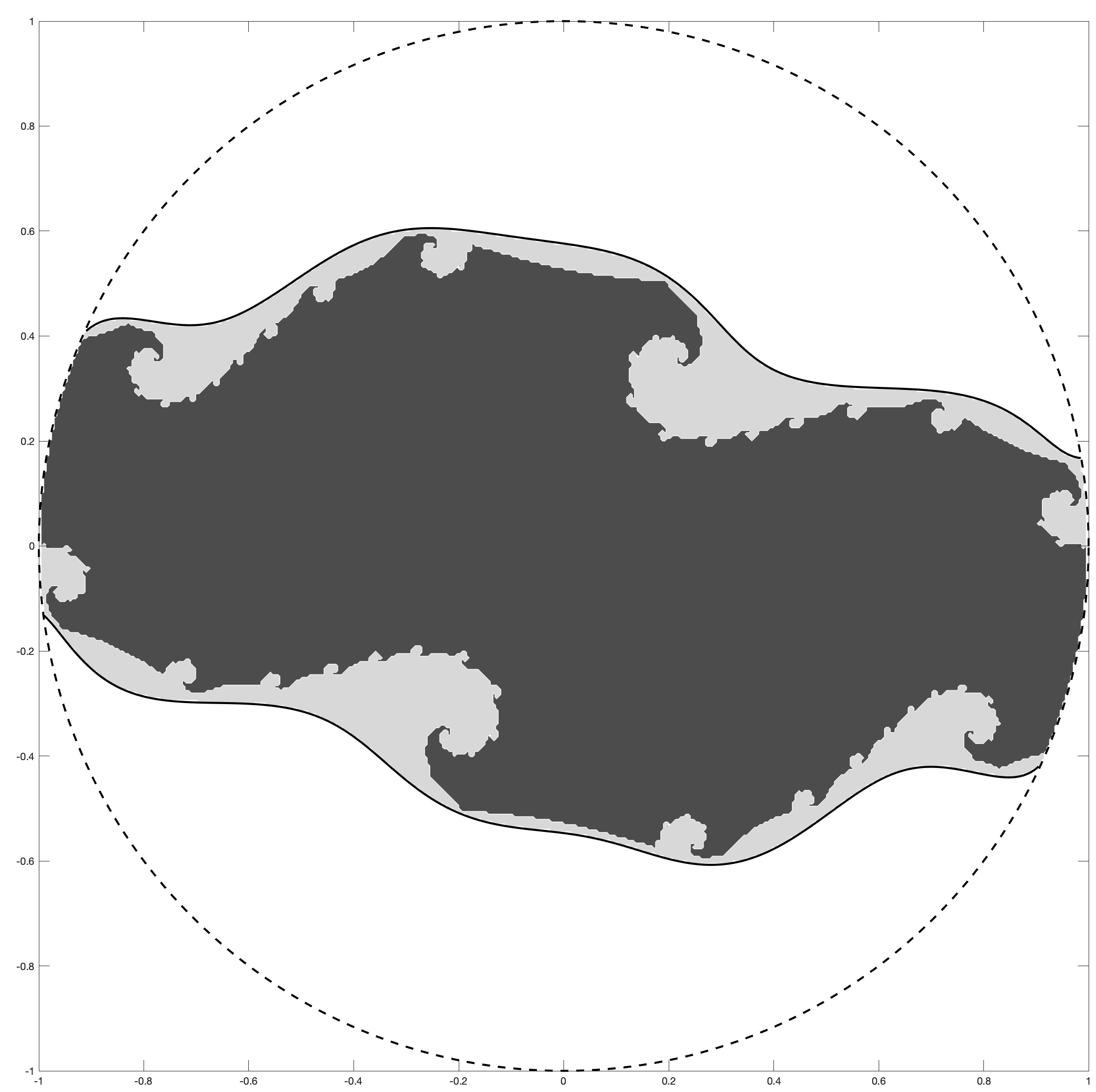}}
\put(250,0){\includegraphics[width=60mm]{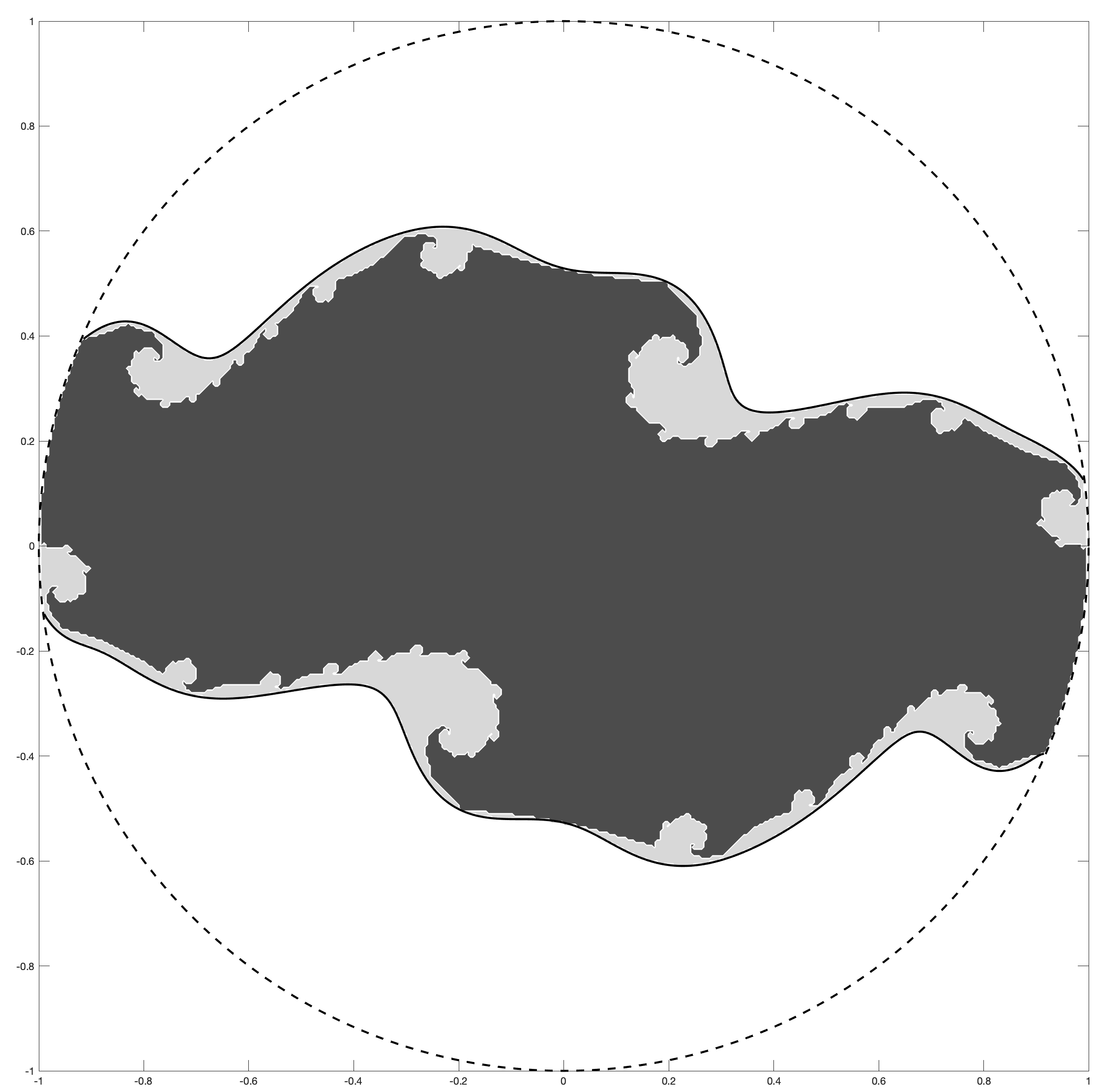}}

\put(105,145){\scriptsize Monomials: $d = 10$}
\put(304,145){\scriptsize Monomials: $d = 18$}

\end{picture}
\caption{\small \textbf{Julia map:} approximations obtained from the data-driven approach with monomial basis. Dark grey: true MPI set. Light grey: outer apporoximation $\Xf_{N,K}$ from~(\ref{eq:XNK}).}
\label{fig:JuliaMonom}
\end{figure*}

\begin{figure*}[h]
\begin{picture}(140,175)
\put(50,0){\includegraphics[width=60mm]{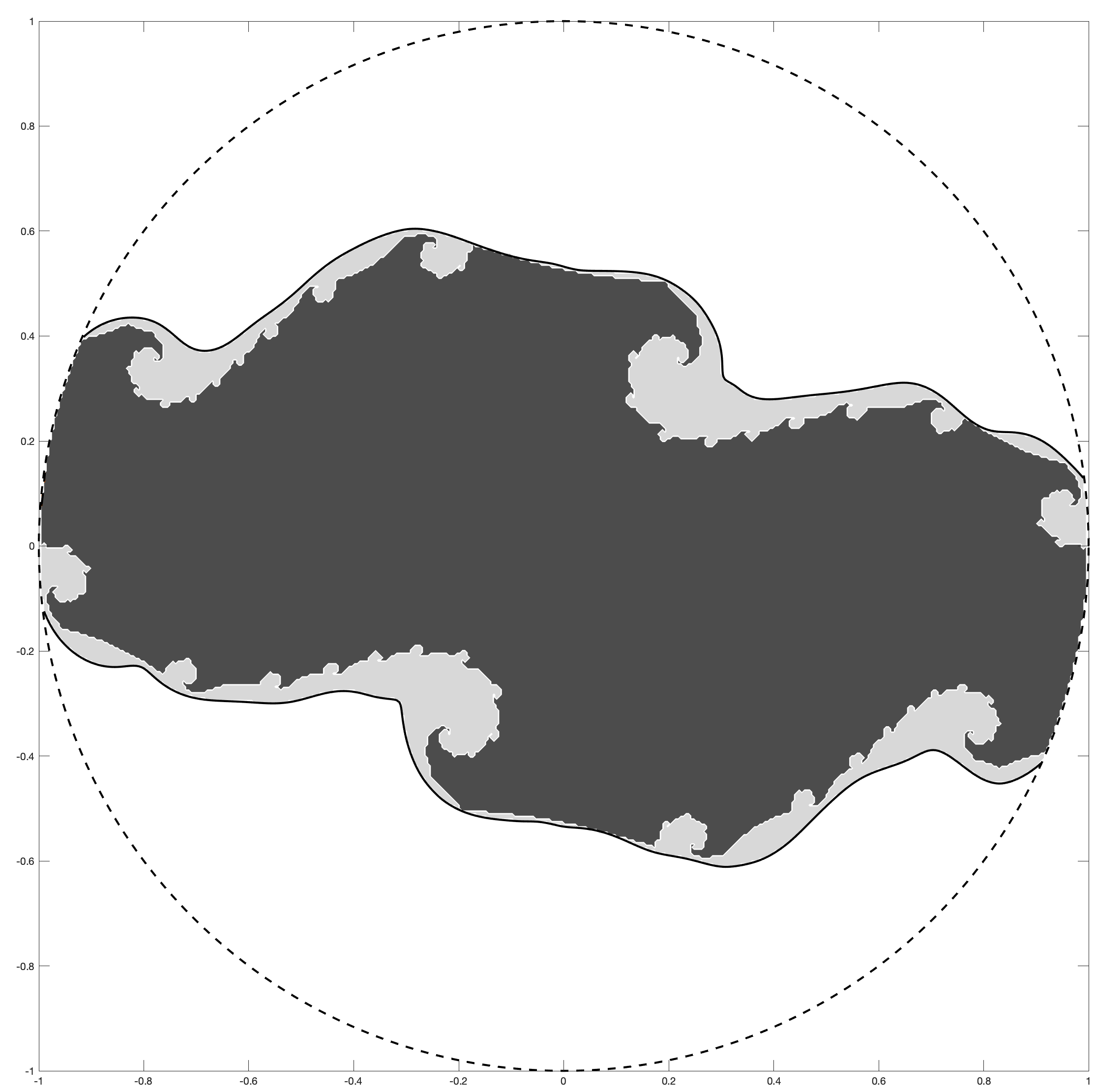}}
\put(250,0){\includegraphics[width=60mm]{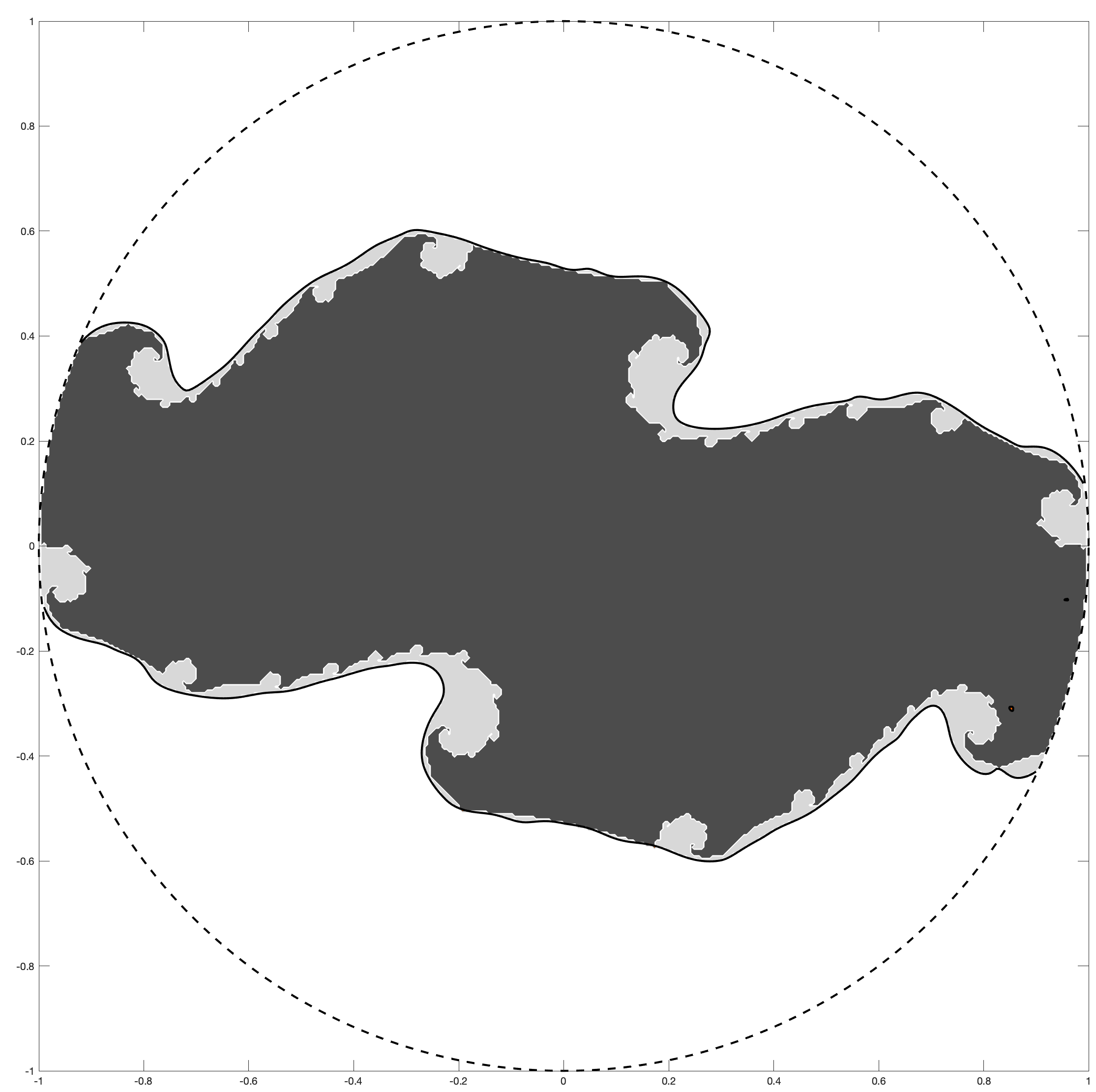}}
\put(110,145){\scriptsize RBFs: $N = 200$}
\put(310,145){\scriptsize RBFs: $N = 600$}
\end{picture}
\caption{\small \textbf{Julia map:} approximations obtained from the data-driven approach with radial basis functions. Dark grey: true MPI set. Light grey: outer apporoximation $\Xf_{N,K}$ from~(\ref{eq:XNK}).}
\label{fig:JuliaRBF}
\end{figure*}

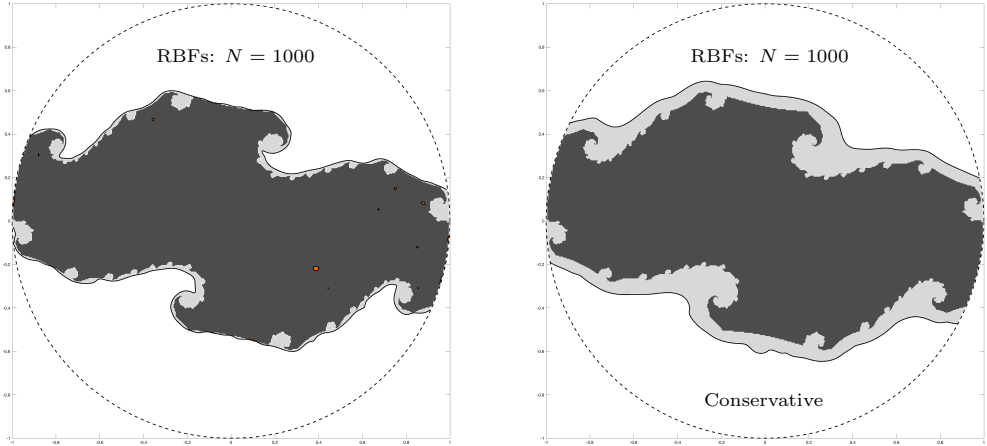
\begin{figure*}[h]
\begin{picture}(140,175)
\put(50,0){\includegraphics[width=60mm]{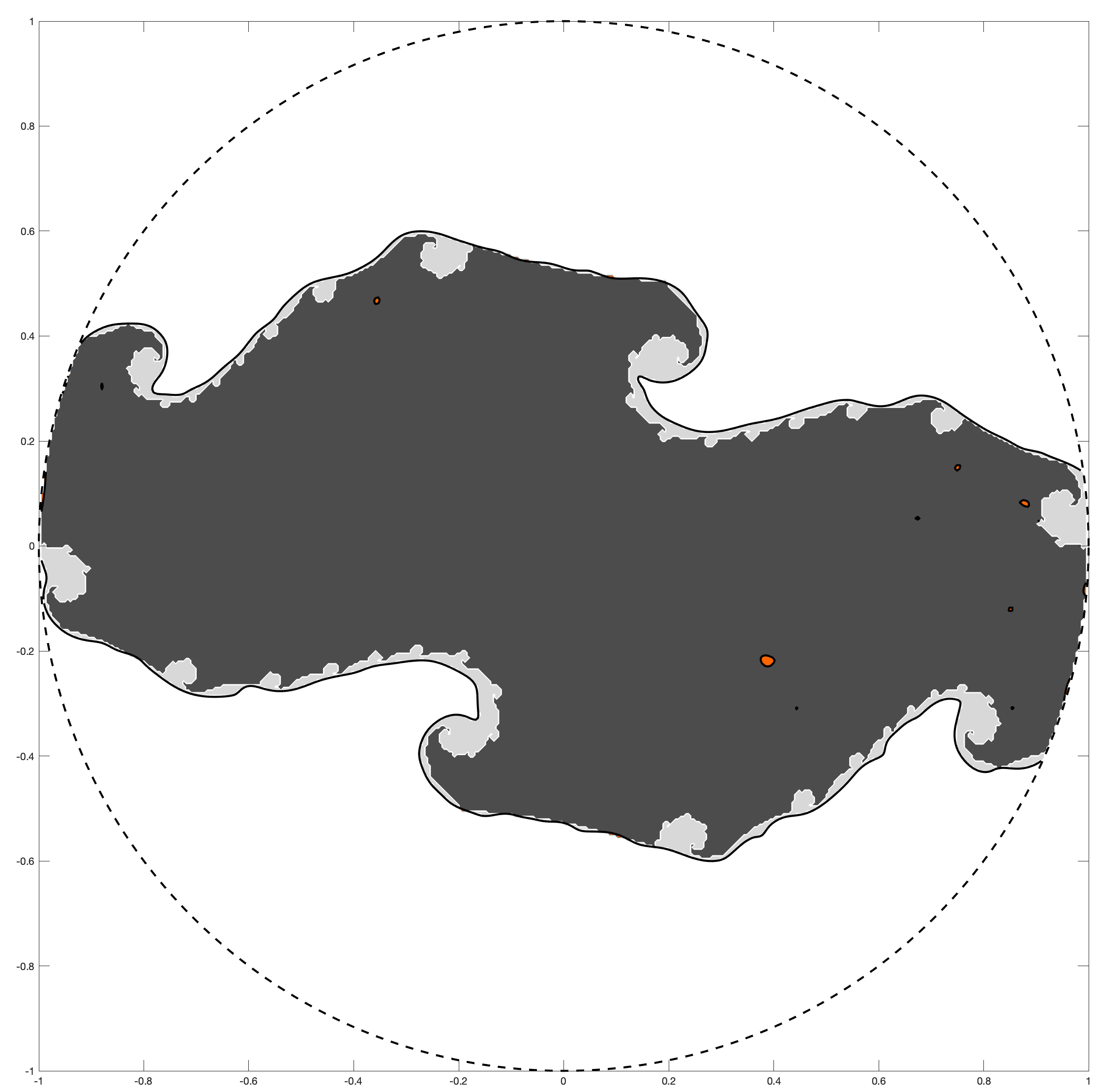}}
\put(250,0){\includegraphics[width=60mm]{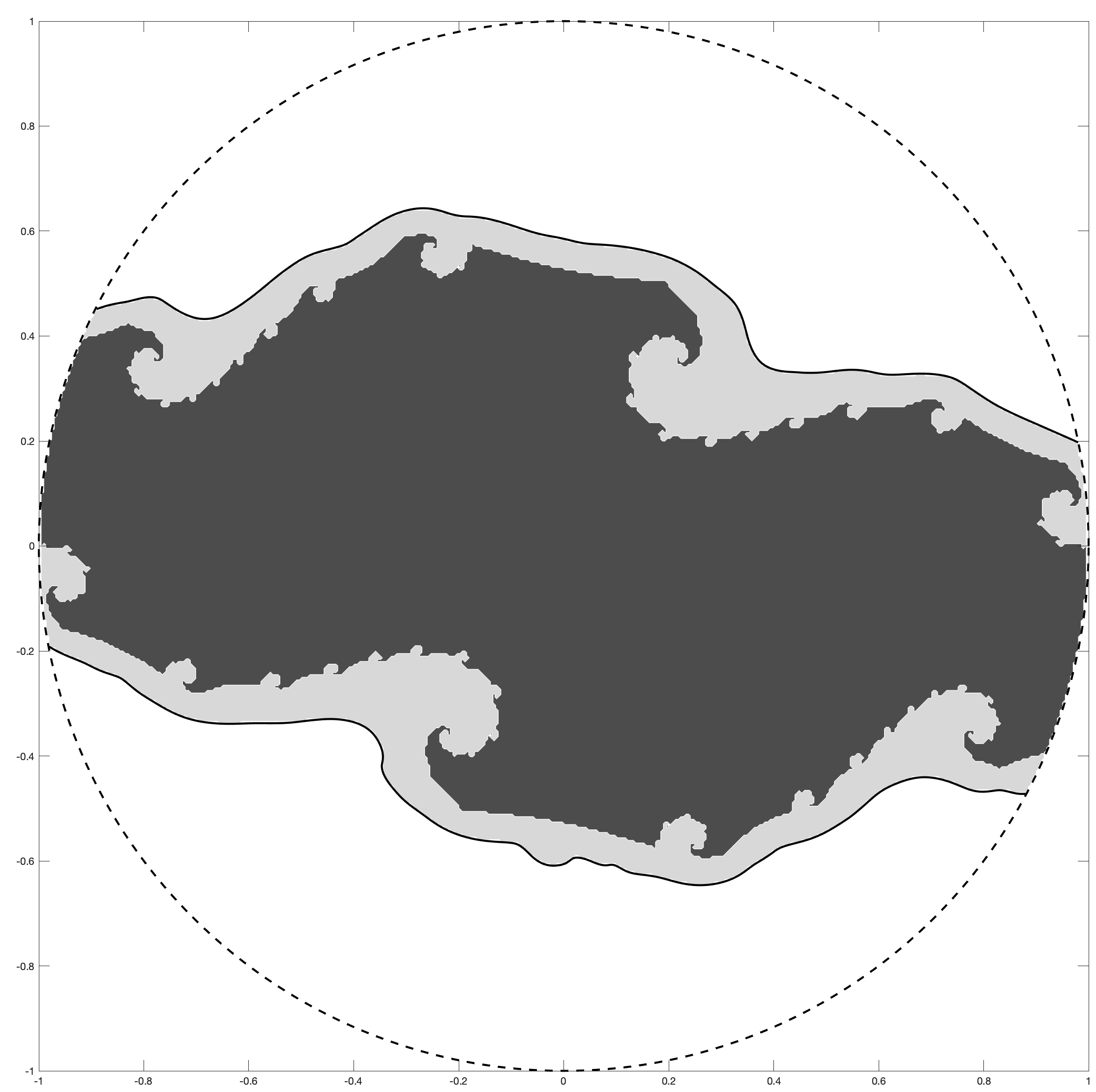}}
\put(110,145){\scriptsize RBFs: $N = 1000$}
\put(310,145){\scriptsize RBFs: $N = 1000$}
\put(315,15){\scriptsize Conservative}
\end{picture}
\caption{\small \textbf{Julia map:} approximations obtained from the data-driven approach with radial basis functions for $N = 1000$. Dark grey: true MPI set. Light grey: outer apporoximation. Orange: misclassified points. Left: non-conservative approximation $\Xf_{N,K}$ from~(\ref{eq:XNK}). Right: conservative approximation $\Xf_{N,K}^{\mr{C}}$ from~(\ref{eq:conservApprox}).}
\label{fig:misclass}
\end{figure*}

\begin{figure*}[h]
\begin{picture}(140,300)
\put(-13,160){\includegraphics[width=65mm]{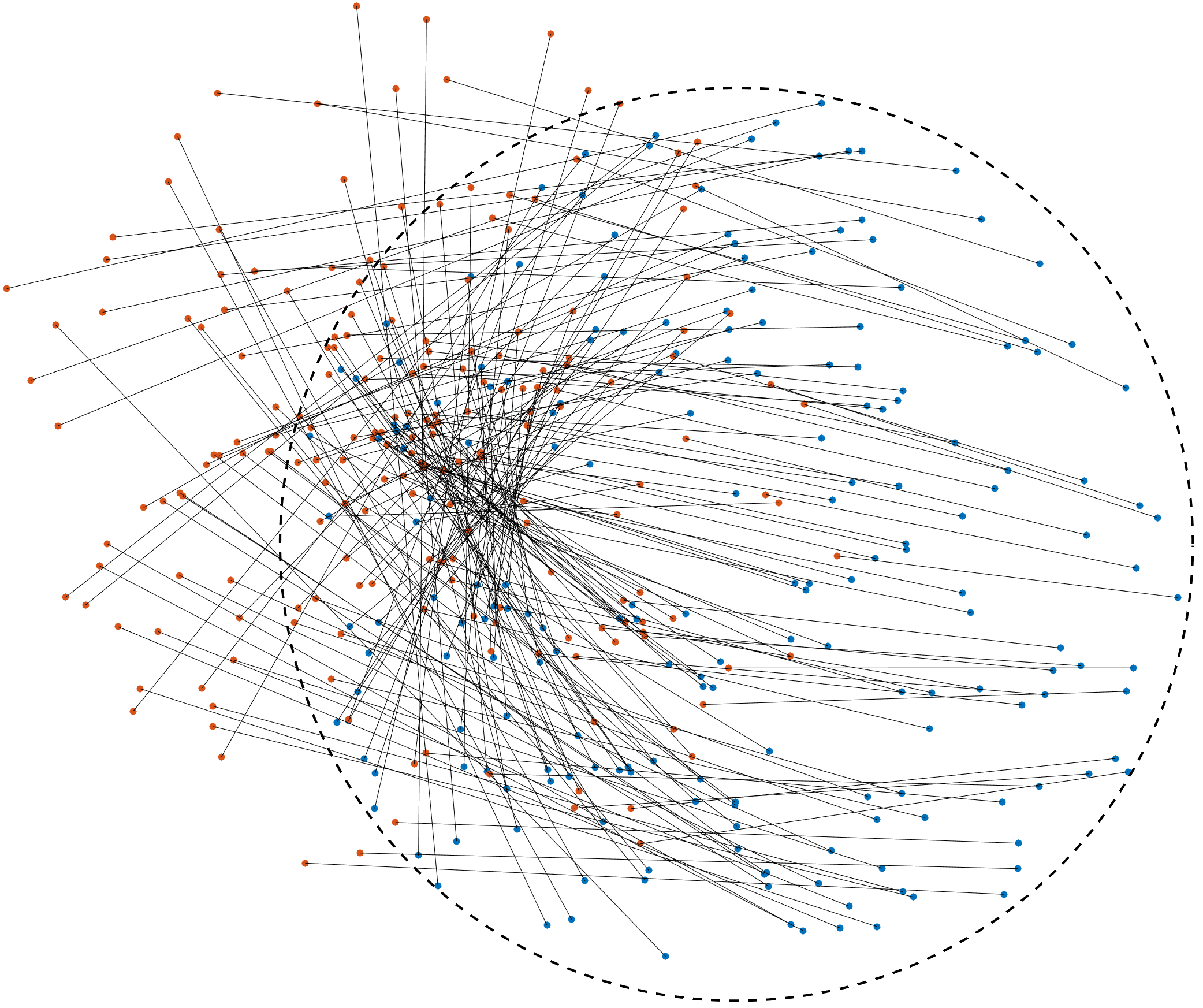}}
\put(175,160){\includegraphics[width=50mm]{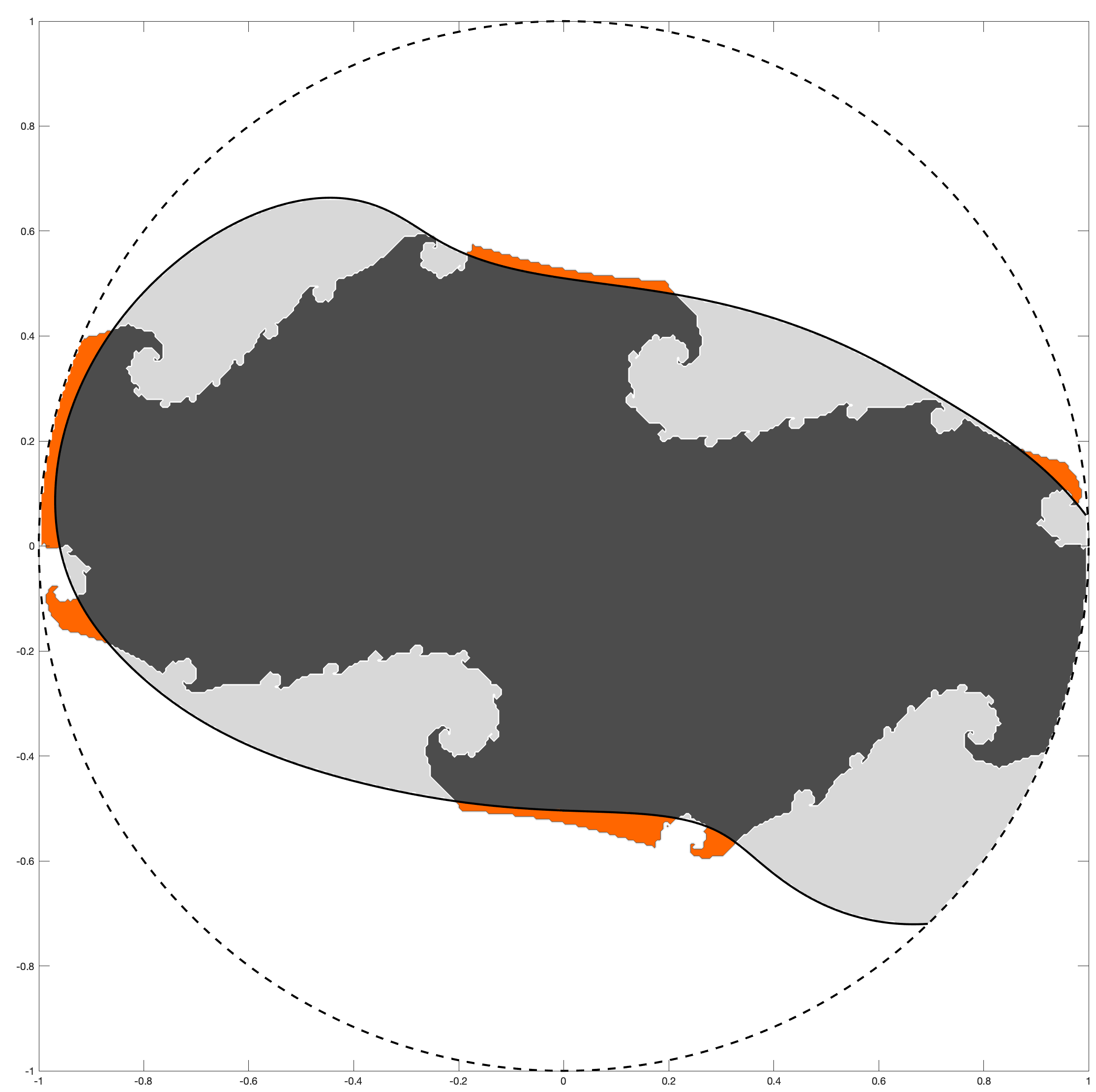}}
\put(323,160){\includegraphics[width=50mm]{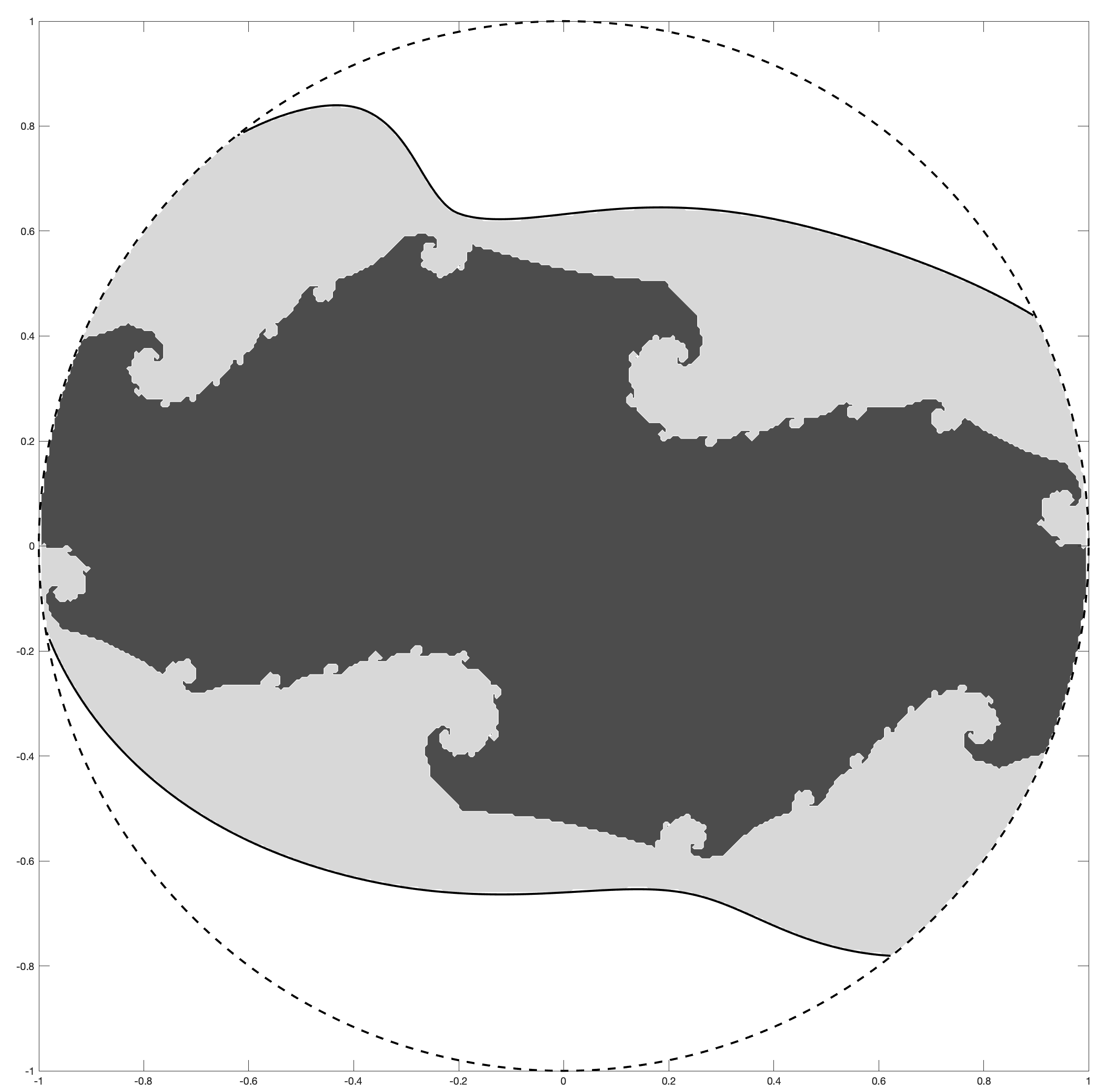}}

\put(220,280){\scriptsize RBFs: $N=15$}
\put(380,280){\scriptsize RBFs: $N=10$}

\put(-13,0){\includegraphics[width=65mm]{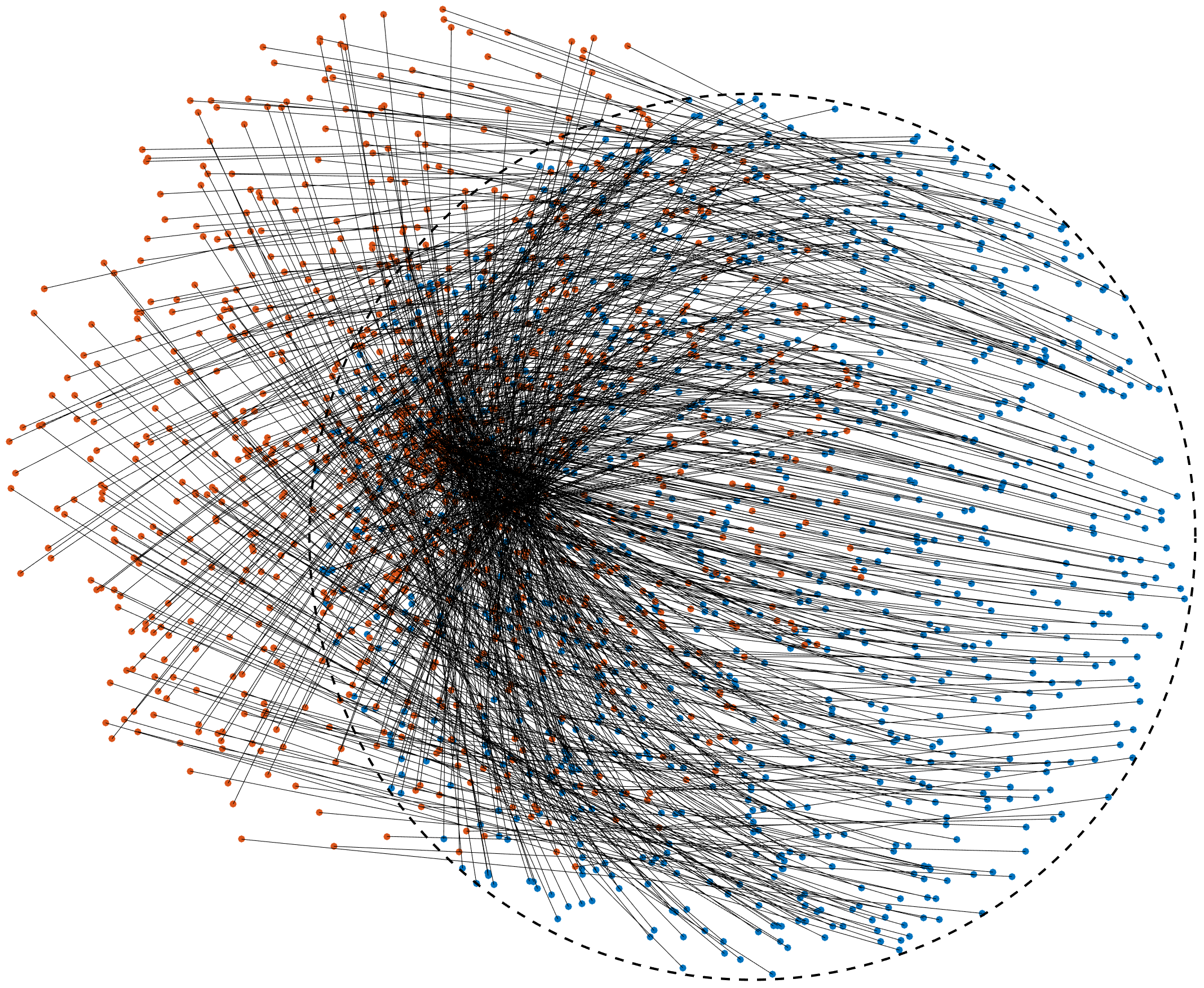}}
\put(175,0){\includegraphics[width=50mm]{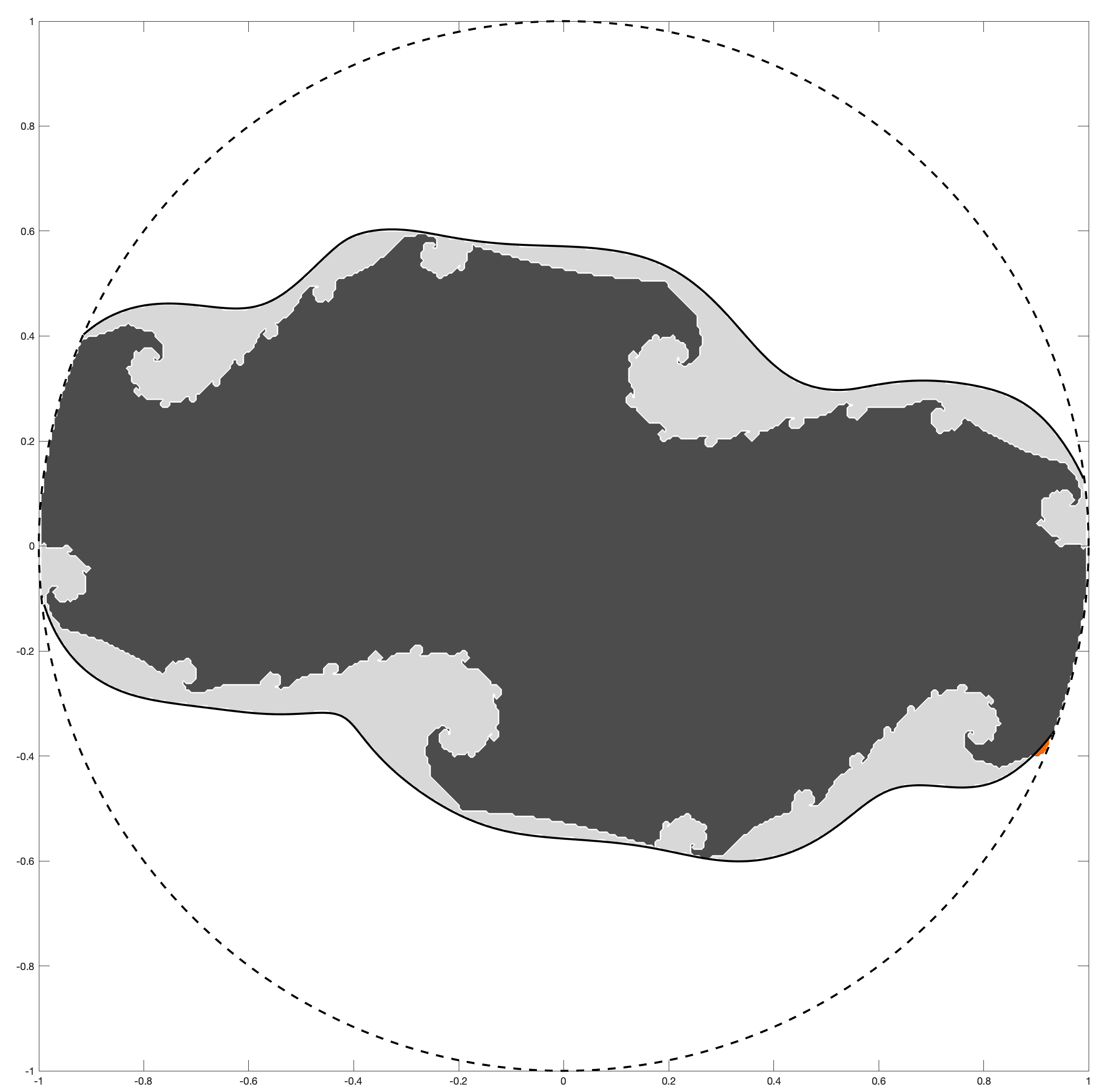}}
\put(323,0){\includegraphics[width=50mm]{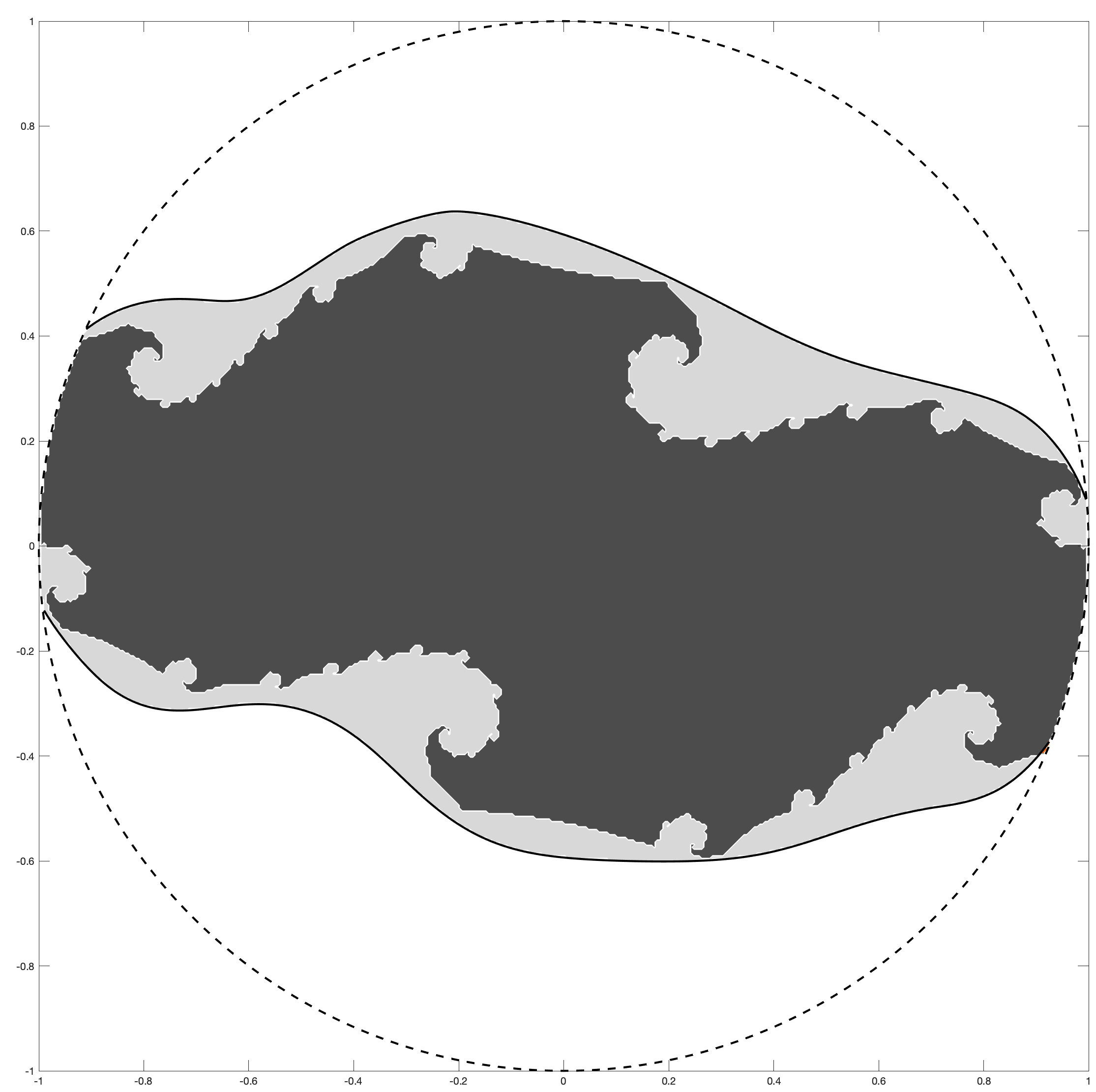}}

\put(220,120){\scriptsize RBFs: $N=40$}
\put(370,120){\scriptsize RBFs: $N=30$}

\end{picture}
\caption{\footnotesize \textbf{Julia map:} Approximations obtained from the data-driven approach with RBFs for low number of data samples. The samples $(x_i,x_i^+)_{i=1}^K$ are depicted in the first column with blue corresponding to $x_i$ and orange to $x_i^+$. Top row: 200 samples. Bottom row: 1000 samples. Dark grey: true MPI set. Light grey: outer apporoximation $\Xf_{N,K}$ from~(\ref{eq:XNK}). Orange: misclassification.}
\label{fig:JuliaLowSamp}
\end{figure*}
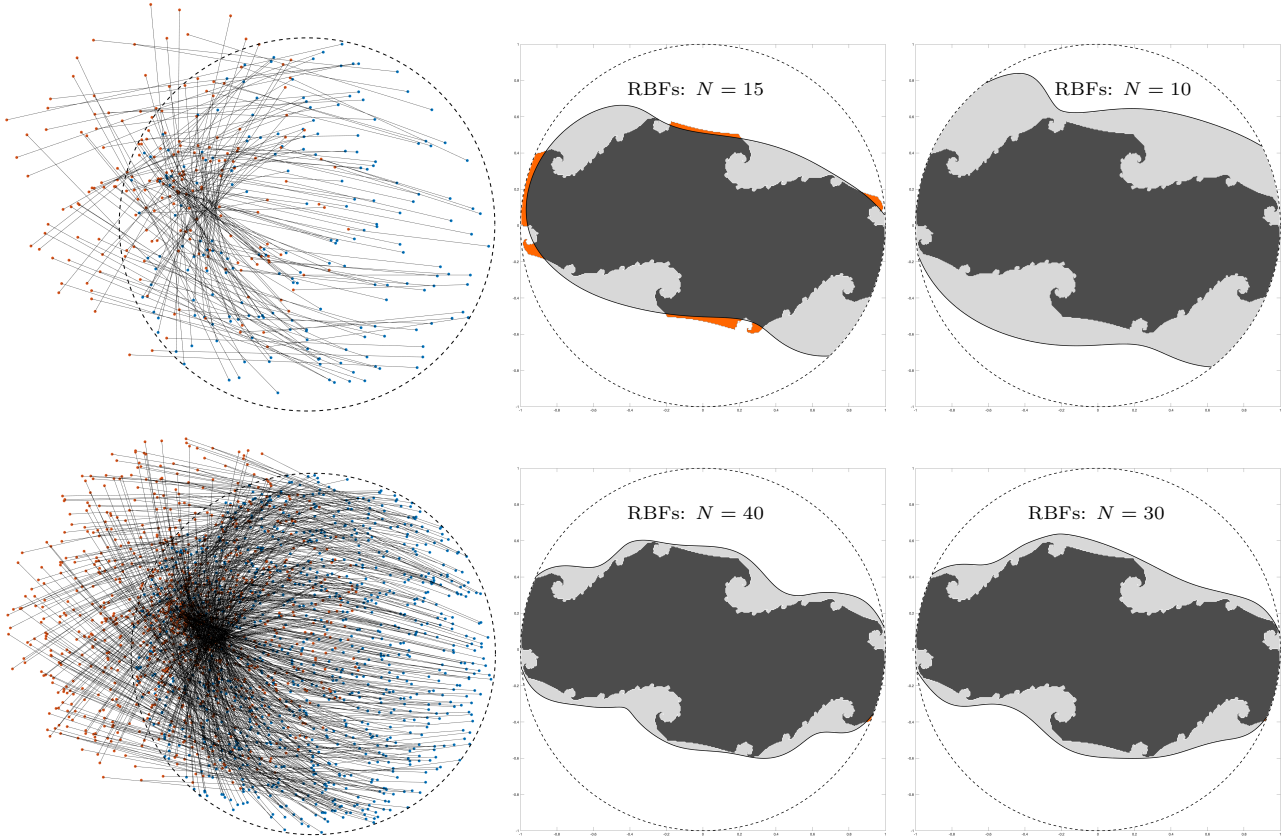

\begin{table}[ht]
\small 
\centering
\caption{\rm \textbf{\footnotesize  Julia Map, monomial basis}: comparison of the proposed Data-driven method and the SDP-based approach of~\cite{kordaMCI}.  The information is reported in the format [Volume error / Misclassificaiton] in percent as defined in~(\ref{eq:volerr}) and (\ref{eq:misclass}) .   }\label{tab:juliaDatavsSDP}\vspace{2mm}
\begin{tabular}{cccccc}
\toprule
Degree  & 10 ($N=66$) & 14 ($N=120$) & 18 ($N=190$) & 20 ($N=231$) \\ \midrule
Data-driven & 20.9 / 0 & 15.0 / 0.0035 & 13.27 / 0.0035 & 11.70 / 0.018 \\\midrule
SDP & 22.42 / 0 & 15.18 / 0  & Num. problems& Num. problems\\
\bottomrule
\end{tabular}
\end{table}

\begin{table}[ht]
\small 
\centering
\caption{\rm \textbf{\footnotesize  Julia Map, radial basis functions}: volume error and misclassificaiton for the proposed data-driven method with radial basis functions.}\label{tab:RBFs}\vspace{2mm}
\begin{tabular}{cccccccc}
\toprule
Size of basis $N$  & 66 & 120 & 200 & 400 & 600 & 1000 \\ \midrule
Volume error [$\%$] & 22.46  & 18.30 &15.27 &10.99 & 9.07 & 7.0  \\\midrule
Misclassification [$\%$] & 0  & 0  & 0.021 &  0.014 & 0.024 & 0.179\\
\bottomrule
\end{tabular}
\end{table}

\subsection{Dimensionality dependence}
In order to investigate scalability and dimensionality dependence of the approach we consider an artificial dynamical system obtained by stacking the Julia map $n/2$ times and making a random unitary coordinate change to couple the states together. We do this in order to be able to vary the dimension of the system while still having access to the true MPI set for comparison. Mathematically, we consider the map $f:\Rb^n\to \Rb^n$ defined by
\[
f = \varphi \circ\underbrace{(f_\mr{julia},\ldots,f_\mr{julia} )}_{n/2\ \mr{times}} \;\circ\; \varphi^{-1},
\]
where $f_\mr{julia} : \Rb^2\to\Rb^2$ is the right-hand-side of~(\ref{eq:julia}) and $\varphi(x) = \mathbf{Q}x$ with $\mathbf{Q} \in \Rb^{n\times n}$ being a random unitary matrix. The state constraint set is set to be $\Xf = \varphi([-1,1]^n)$, resulting in the MPI set being equal to $\varphi^{-1}(\underbrace{\Xf_{\infty,\mr{julia}} \times \ldots \times \Xf_{\infty,\mr{julia}}}_{n/2\; \mr{times}})$, where $\Xf_{\infty,\mr{julia}}$ is the MPI set of the Julia map with  the constraint $[-1,1]^2$. We fix the number of radial basis functions with randomly generated centers over $\Xf$ to $N = 1600$ and we use $3\cdot 10^4$ data samples.  Figure~\ref{fig:JuliaDim} depicts the results for $n\in \{4,6,8,10\}$. As expected, we observe that the approximations get less tight as the dimension increases; however, we do not observe an increase in the misclassification. An intuitive explanation for this is that the thin-plate spline radial basis functions are not spatially localized (as opposed to, e.g., Gaussian radial basis functions with narrow width) and hence not as prone to large inter-sample errors.

\begin{figure*}[h]
\begin{picture}(140,300)
\put(50,175){\includegraphics[width=60mm]{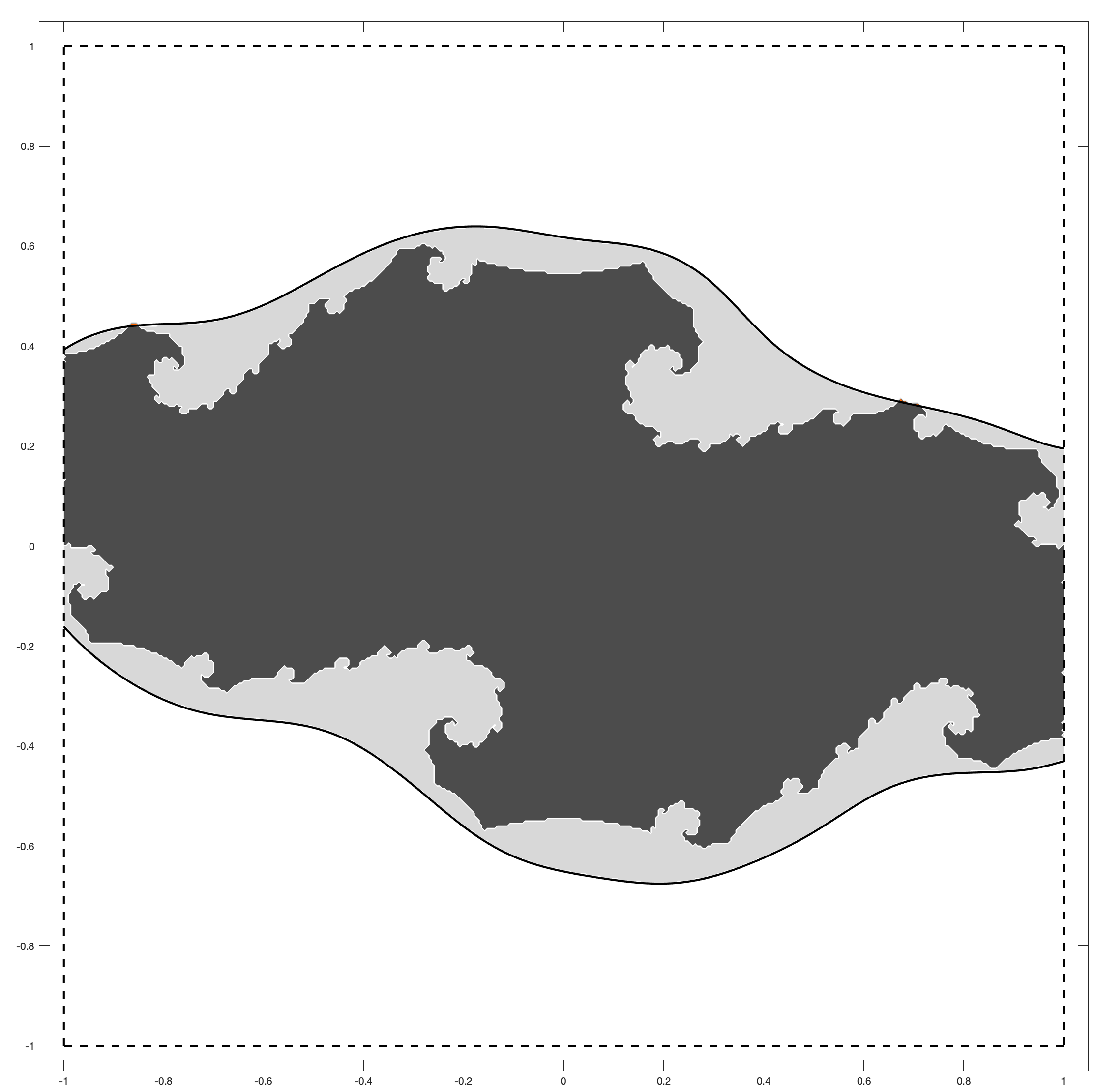}}
\put(250,175){\includegraphics[width=60mm]{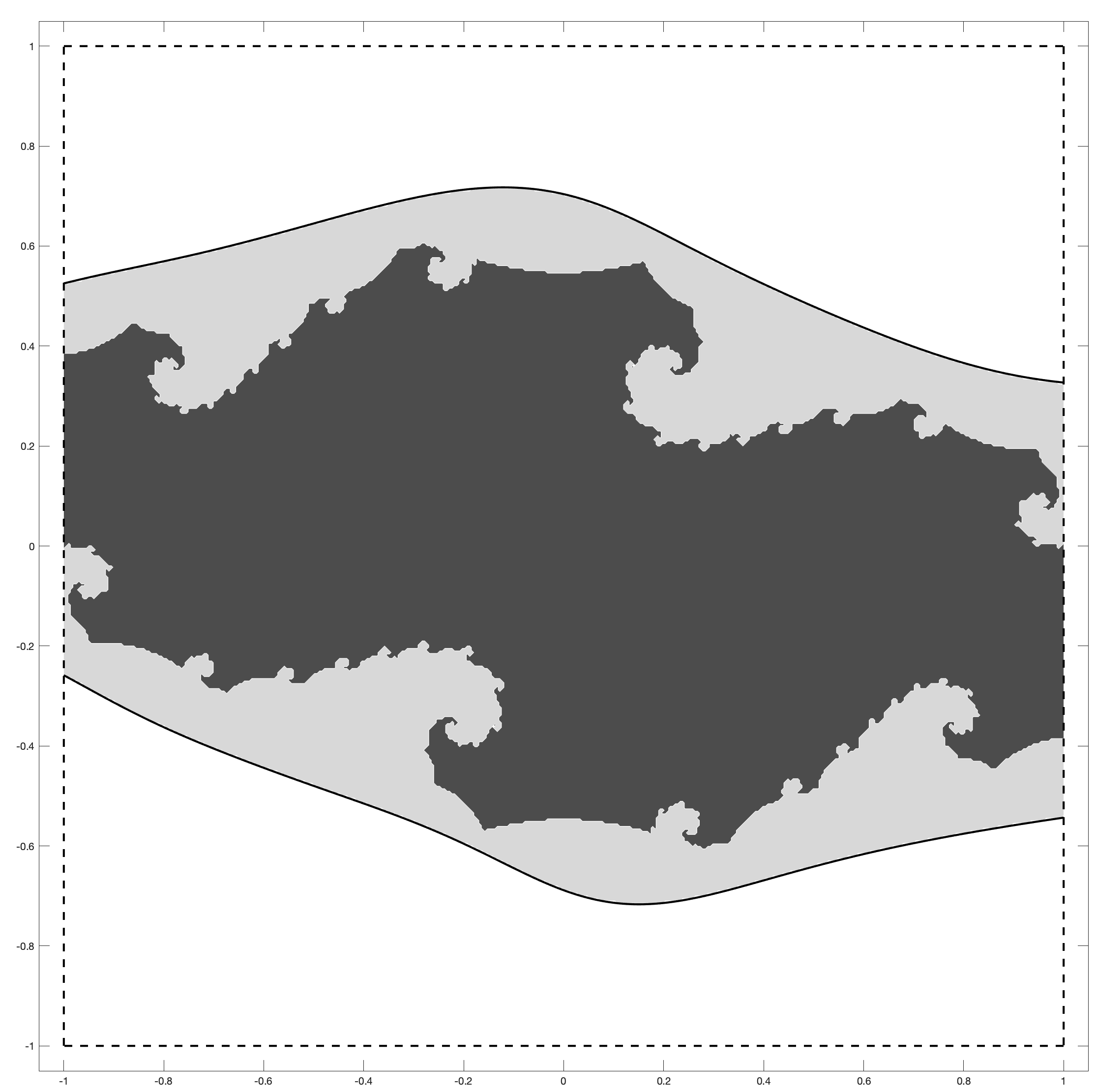}}
\put(50,0){\includegraphics[width=60mm]{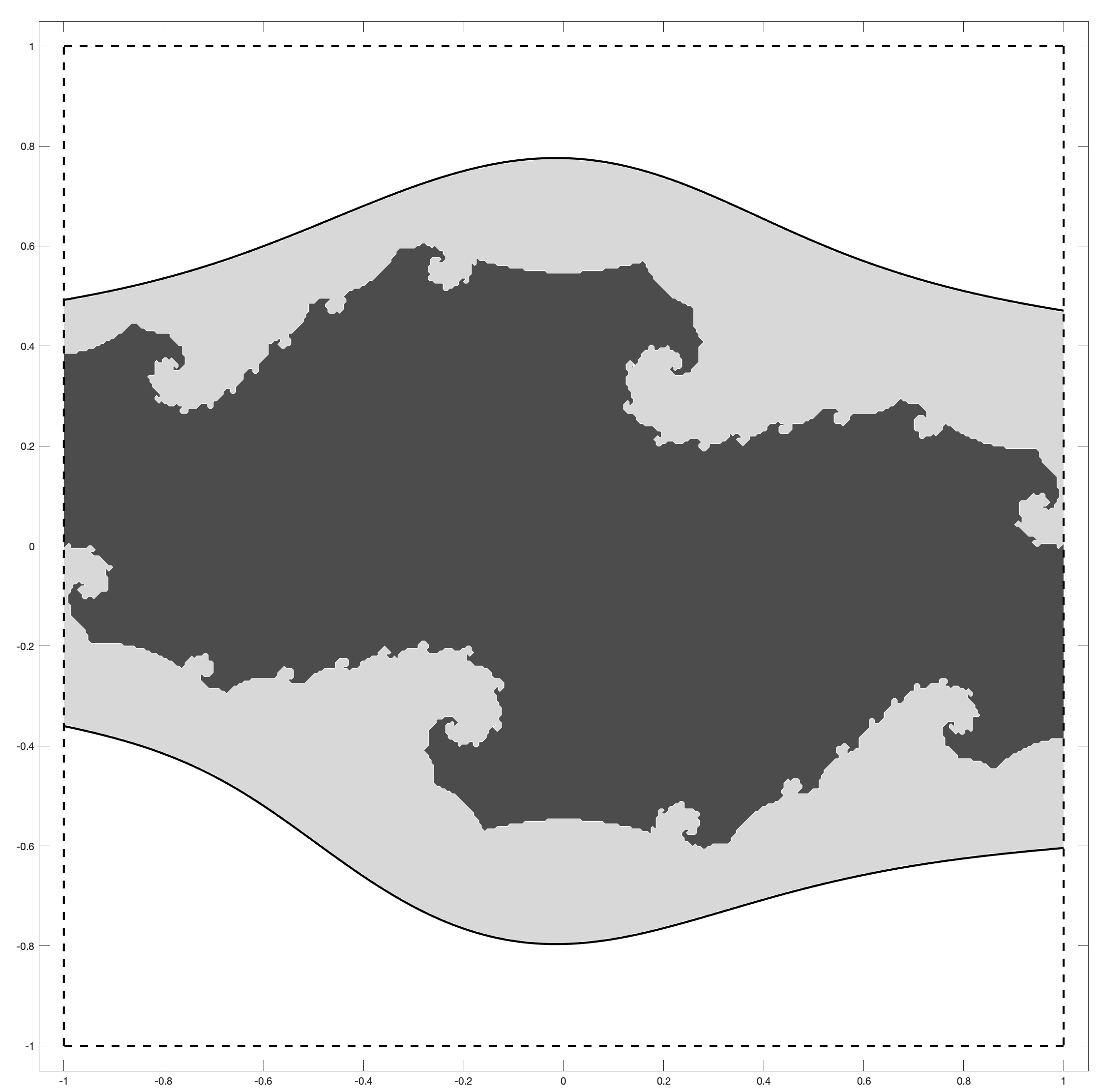}}
\put(250,0){\includegraphics[width=60mm]{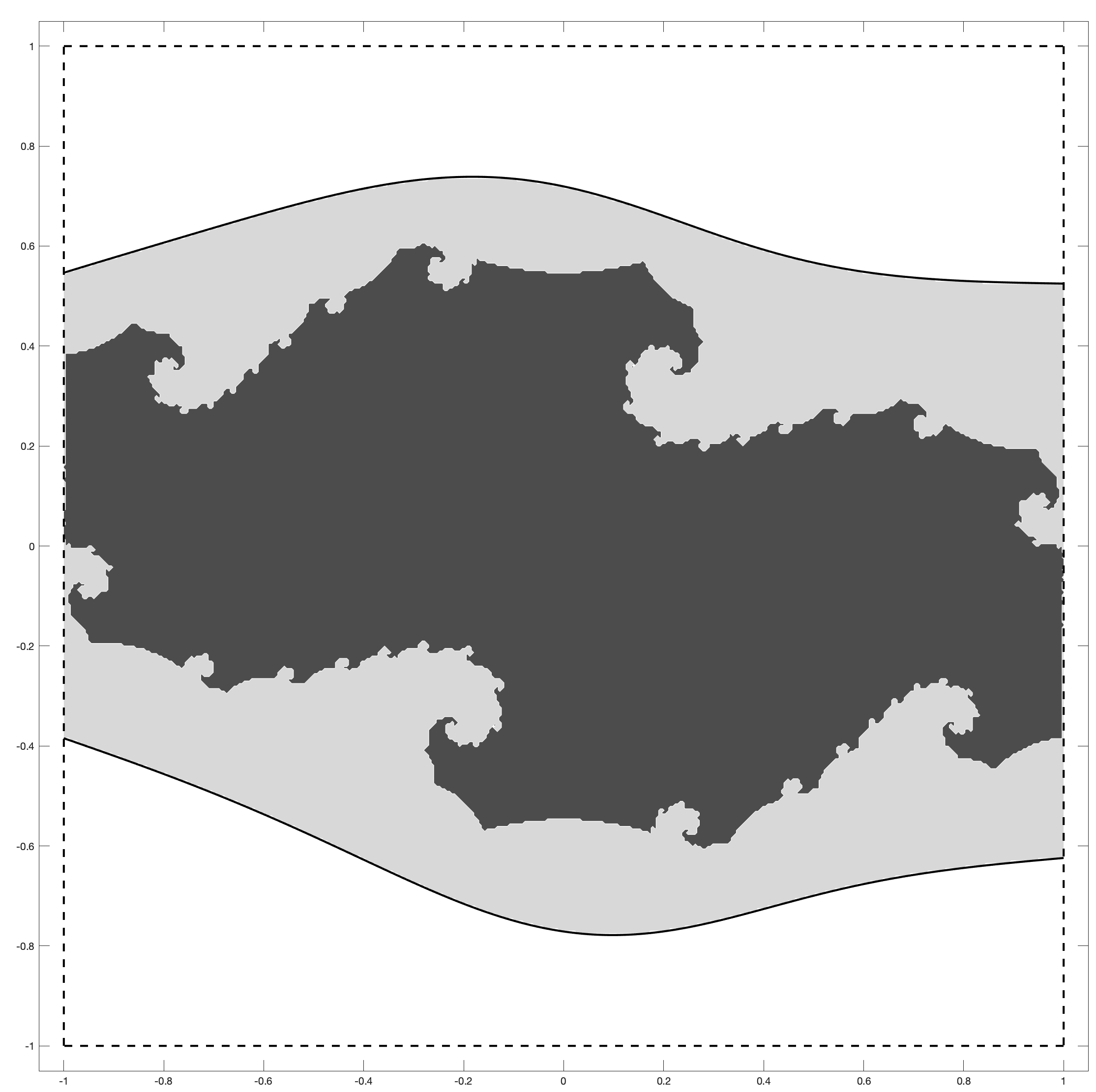}}
\put(92,325){\scriptsize State-space dimension = 4}
\put(292,325){\scriptsize State-space dimension = 6}
\put(92,151){\scriptsize State-space dimension = 8}
\put(292,151){\scriptsize State-space dimension = 10}
\end{picture}
\caption{\small \textbf{Dimensionality dependence:} Outer approximations with thin-plate spline 1600 RBFs for the product of Julia maps. Dark grey: projection of the MPI set on the first two coordinate axes. Light grey: projection of the outer approximation on the first two coordinate axes.}
\label{fig:JuliaDim}
\end{figure*}


\subsection{Controlled example}
In order to demonstrate the approach in a controlled setting, we consider the three-dimensional H\'enon map with control from~\cite{kordaMCI} with the dynamics
\begin{align*}
 x_1^+ &= 0.44 - 0.1x_3 - 4x_2^2 + 0.25u,\\
 x_2^+ &= x_1-4x_1x_2,\\
 x_3^+ &= x_2
\end{align*}
and subject to the constraints $\Xf = [-1,1]^3$ and $\Uf = [-1,1]$. We use 1000 thin-plate spline radial basis functions with centers generated randomly in $\Xf$ and $5\cdot 10^4$ data samples $(x,x^+)$ \new{sampled uniformly at random over $\Xf$ with the associated control inputs sampled uniformly at random over $\Uf$;  the control samples are not used in the algorithm.} The discount factor $\alpha$ is set to 0.2. The true MCI set is unknown for this example; therefore we compare the approximations~(\ref{eq:XNK}) with the conservative approximations~(\ref{eq:conservApprox}). Figure~\ref{fig:HenonCont} shows the results. We observe that the approximation~(\ref{eq:XNK}) captures smaller inner structures compared  to the conservative approximation.

\begin{figure*}[h]
\begin{picture}(140,160)
\put(30,-10){\includegraphics[width=85mm]{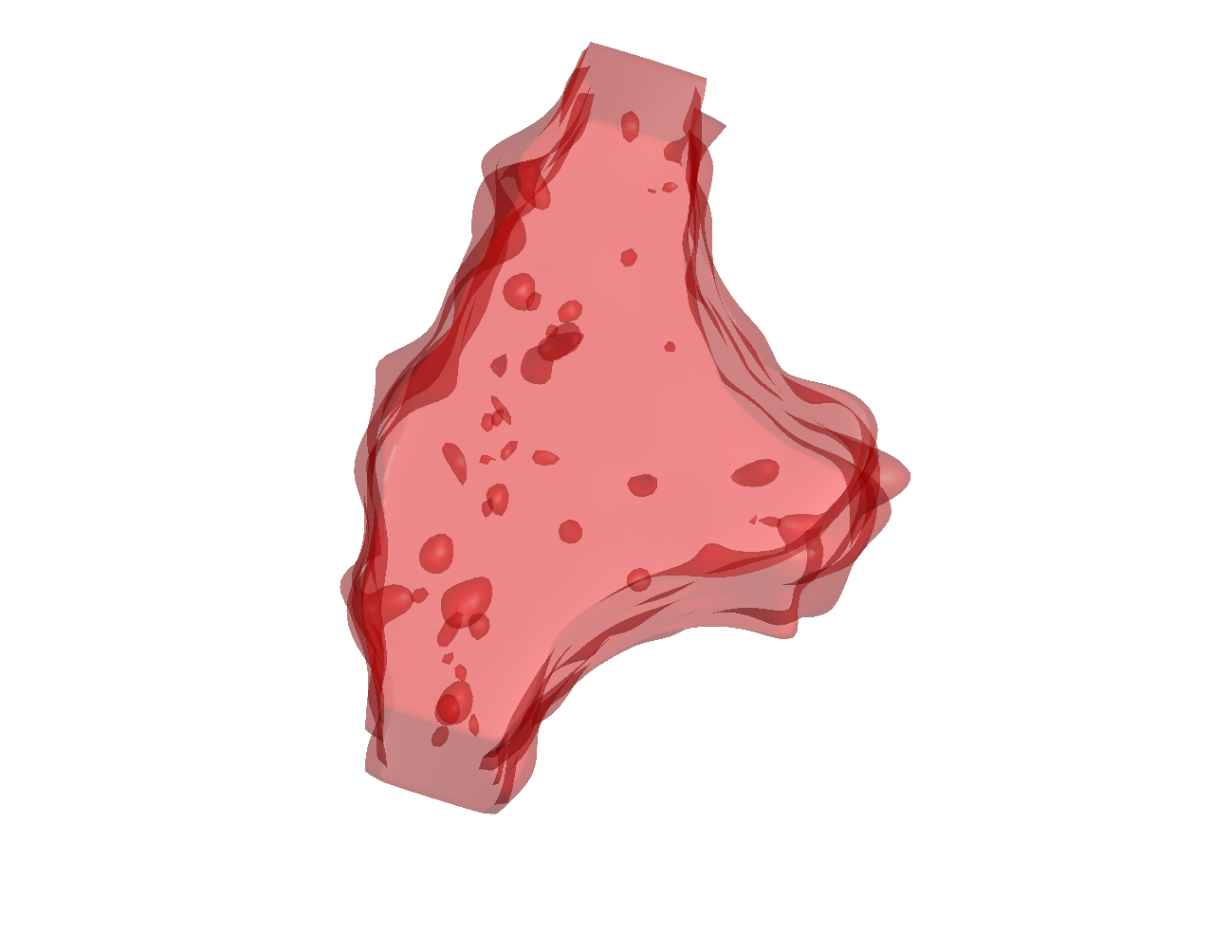}}
\put(230,-10){\includegraphics[width=85mm]{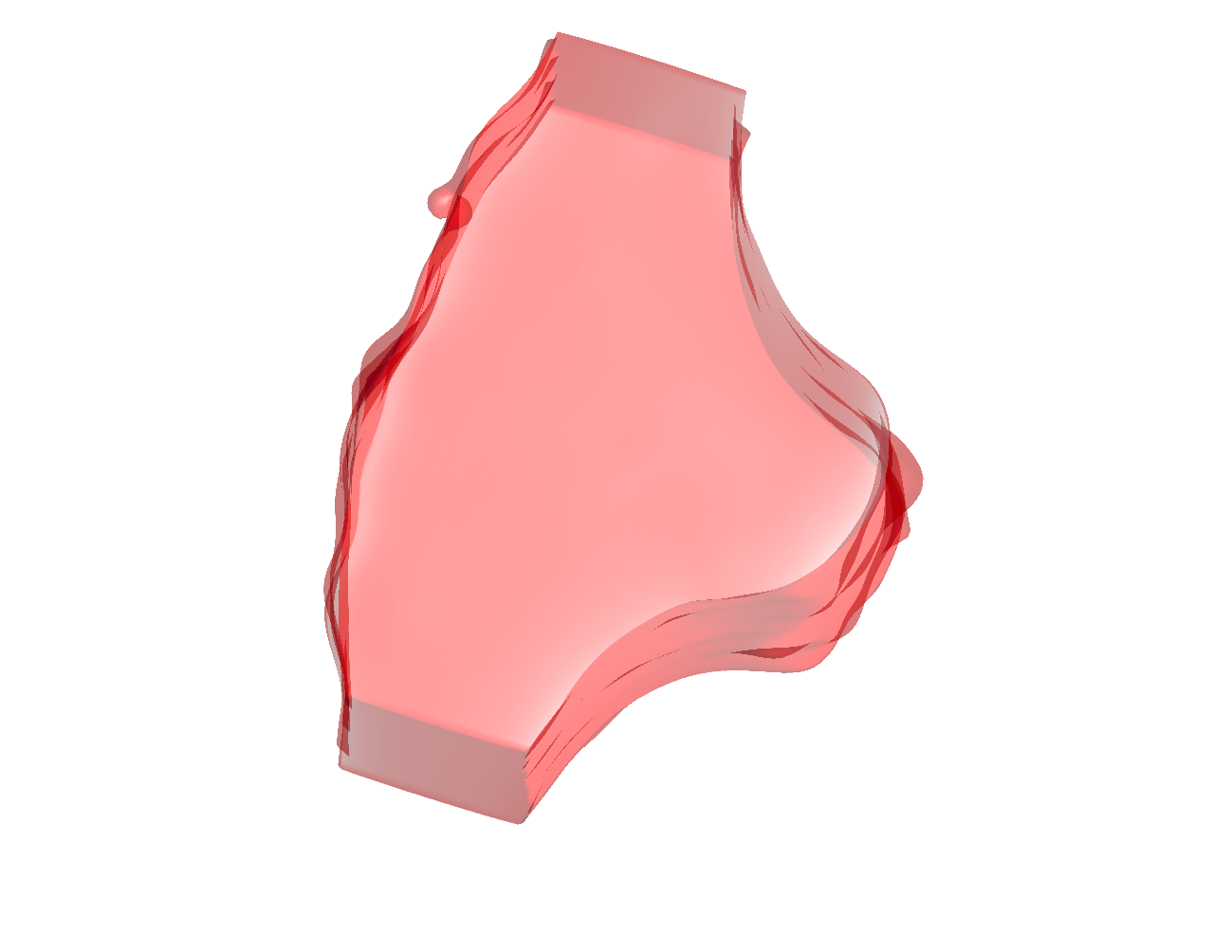}}
\end{picture}
\caption{\small \textbf{Controlled H\'enon map}: Comparison of the ``standard'' (left) approximation~(\ref{eq:XNK}) and the conservative (right) approximation~(\ref{eq:conservApprox}). The true MCI set is unknown in this case. }
\label{fig:HenonCont}
\end{figure*}

\subsection{Switched system}
In our last example, we demonstrate the applicability for systems with a discontinuous transition mapping~$f$. We consider the system
\[
\dot{x} = \begin{cases}
\begin{bmatrix} - 1 & 1 \\ -5  & -0.1\end{bmatrix}\varphi(x) & x_1^2 \le x_2^2  \vspace{3mm}  \\ 
\begin{bmatrix}  -0.1 & 5 \\ -1  & -0.1\end{bmatrix}\varphi(x) & x_1^2 > x_2^2\;,  \\
\end{cases}
\]

where $\varphi$ is a given, possibly nonlinear, mapping. For $\varphi(x) = x$ we obtain the classical flower system from~\cite{johansson1997computation}. The discrete-time dynamics $f$ investigated is the Runge-Kutta four discretization of this continuous-time ordinary differential equation with sampling interval~$0.05$. The constraint set is $\Xf = [-1,1]^2$. We utilize 600 thin-plate-spline RBFs with randomly selected centers, $3\cdot10^4$ data samples and discount factor $\alpha = 0.8$. Figure~\ref{fig:switched} shows the results for $\varphi = x$ (i.e., resulting in a switched affine system) and for $\varphi = [\sin x_1^3,\sin x_2^3]^\top$ (i.e., resulting in a switched nonlinear system). We observe that the approach can provide valid outer approximations even in this setting to which some of the theoretical results do not apply due to the lack of Lipschitz continuity of $f$.

\begin{figure*}[h]
\begin{picture}(140,175)
\put(50,0){\includegraphics[width=60mm]{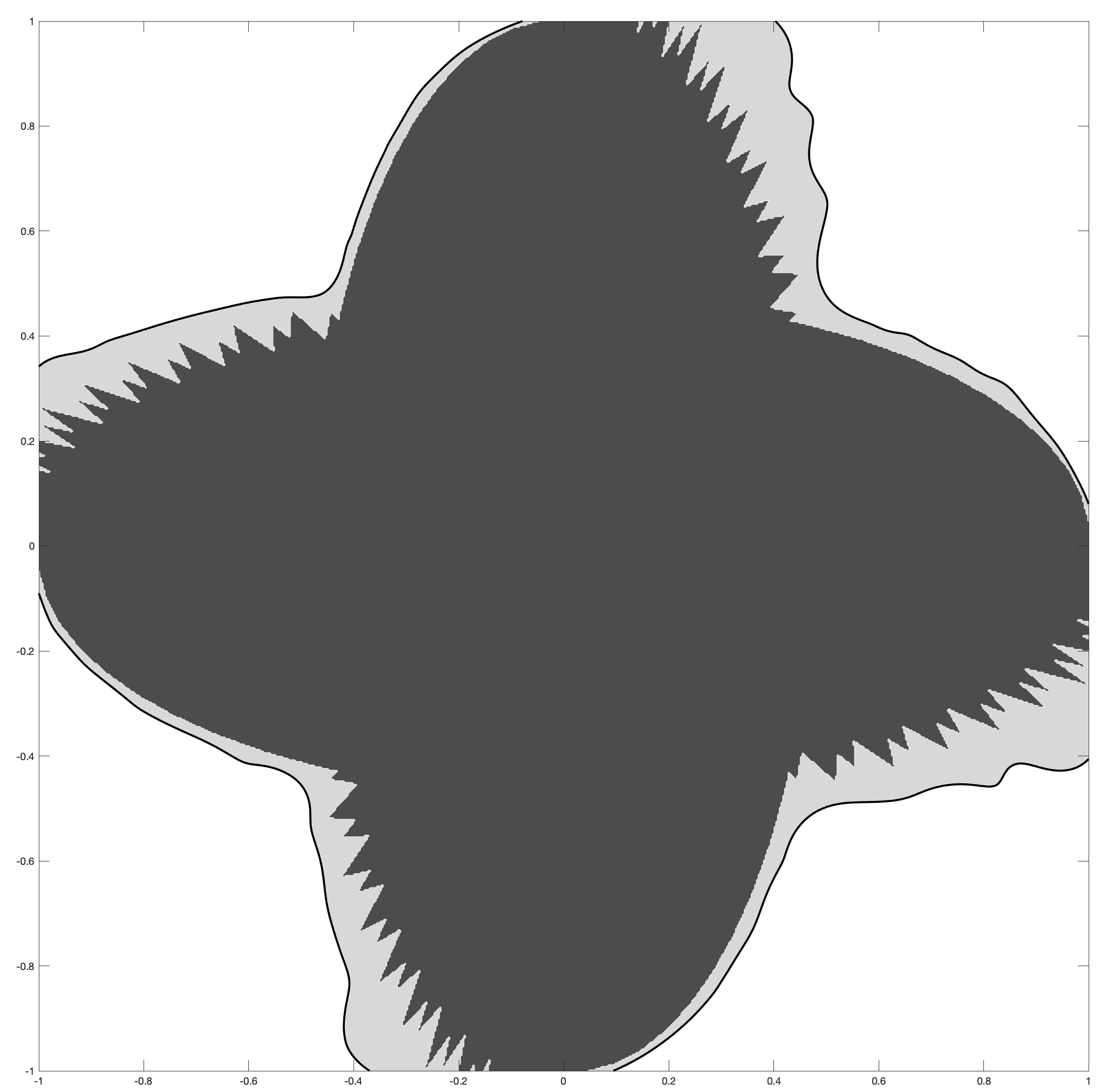}}
\put(250,0){\includegraphics[width=60mm]{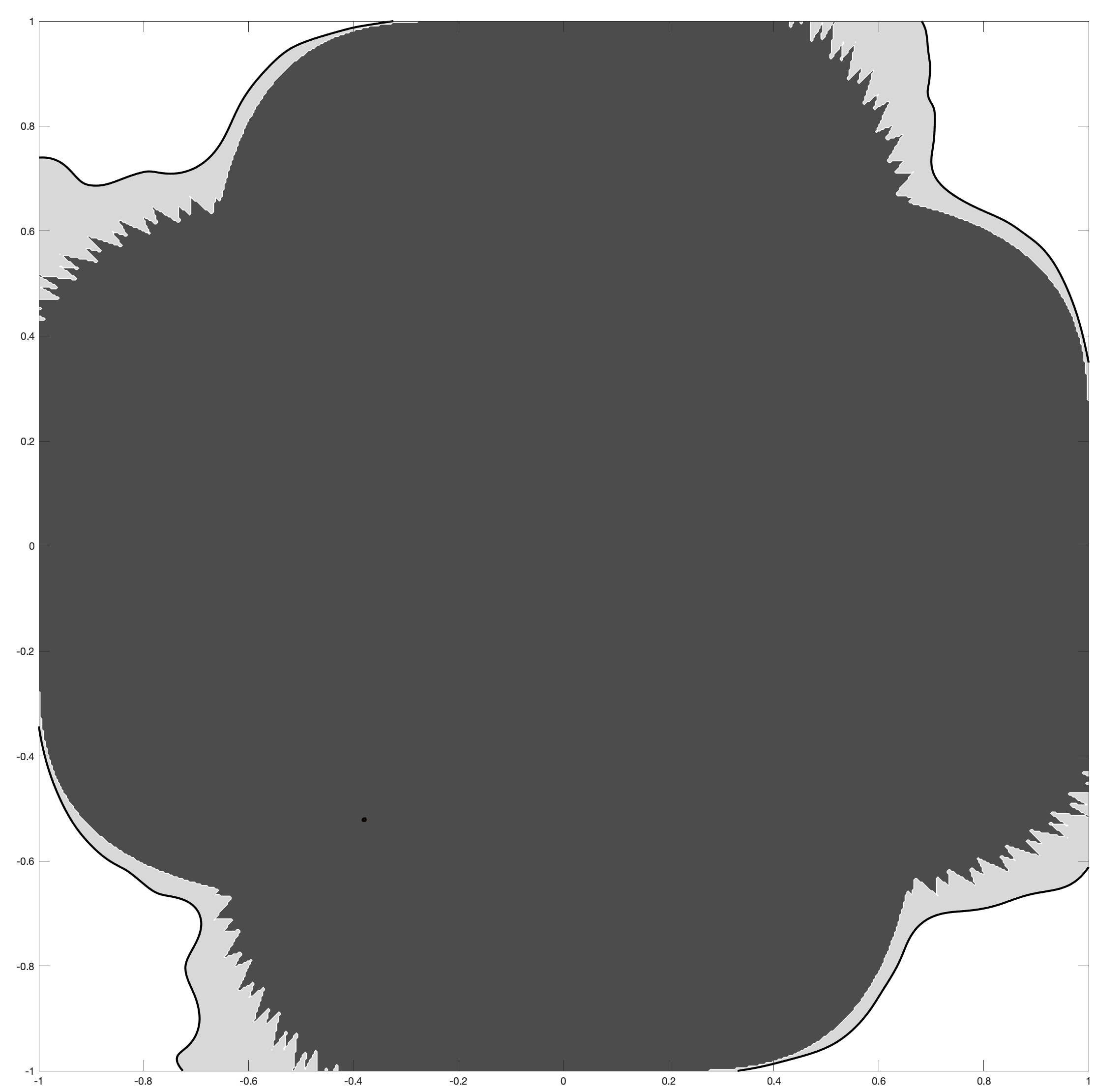}}
\end{picture}
\caption{\footnotesize \textbf{Switched system}: Left: switched affine system ($\varphi(x) = x$). Right: switched nonlinear system ($\varphi(x) = [\sin x_1^3,\sin x_2^3]^\top$). Dark grey: true MPI set. Light grey: outer apporoximation $\Xf_{N,K}$ from~(\ref{eq:XNK}).}
\label{fig:switched}
\end{figure*}

\section{Conclusion and outlook}
We presented a convex-optimization-based data-driven method for approximation of the MPI and MCI sets, furnished with theoretical analysis and detailed numerical experiments. The method is simple and general, relying on the solution to a single finite-dimensional LP with mild underlying assumptions. The worst case complexity upper bounds of Theorem~\ref{thm:mainBound} are exponential in the state-space dimension. We believe that this is unavoidable at this level of generality; a possible future research direction would be to derive matching lower bounds.

Other future research directions include:
\begin{itemize}
\item Structure exploitation. Exploiting problem-specific specific structure such as sparsity or symmetries should increase the efficiency of method, both in terms of the number of basis functions required as well as in terms of the number of data samples.
\item Ex-post validation. Deriving rigorous probabilistic guarantees on the ``conservative'' approximation of~(\ref{eq:conservApprox}) or investigating more sophisticated ex-post validation techniques (e.g., subsampling-based cross-validation) would increase applicability in safety-critical applications. 
\item Adaptive sampling and basis selection. The presented approach is the simplest possible version of the algorithm. More sophisticated algorithms could employ adaptive basis selection and/or adaptive sampling (relevant when samples can be drawn sequentially in a given application).
\end{itemize}

\clearpage

\section{Proofs}\label{sec:proofs}
\paragraph{Proof of Theorem~\ref{thm:mainLP}}

The fact that $v^\star$ is bounded (by $1/(1-\alpha)$) follows from the facts that $|\bar l| \le 1$ on $\Rb^n$ and $\alpha \in (0,1)$. Measurability of $v^\star$ with respect to $x$ is trivial as it is a convergent sum of nonnegative measurable functions. 

The proof of attainment of the supremum in~(\ref{opt:LPinf}) follows a standard dynamic programming argument applied to an infinite-horizon discounted optimal control problem (here without control) with the dynamical system $x^+ = \bar f(x)$, stage cost $\bar l(x) $ and constraint set~$\Xf$. The auxiliary dynamical system $x^+ = \bar f(x)$ was crafted such that the constraint set is invariant under $\bar f$. Now, given any $v\in\Bc(\Xf)$ feasible in~(\ref{opt:LPinf}), we have
\[
v \le \bar l + \alpha v\circ \bar f \quad \mr{on} \;\; \Xf.
\]
Multiplying by $\alpha$ and composing with $\bar f$ (using the invariance of $\Xf$ under $\bar f$) we obtain
\[
\alpha v\circ f \le \alpha \bar l \circ \bar f+ \alpha^2 v\circ \bar f^{(2)} \quad \mr{on} \;\; \Xf
\]
and hence
\[
v \le \bar l + \alpha \bar l \circ\bar f + \alpha^2 v\circ \bar f^{(2)}   \quad \mr{on} \;\; \Xf.
\]
Iterating the procedure, we get
\[
v \le \sum_{k=0}^{N-1}\alpha^k \bar l \circ \bar f^{(k)} + \alpha^N v\circ \bar f^{(N)} \quad \mr{on} \;\; \Xf.
\]
Taking the limit as $N\to \infty$, using the facts that $v$ is bounded $\alpha \in (0,1)$, we conclude that
\[
v \le   \sum_{k=0}^{\infty}\alpha^k \bar l \circ \bar f^{(k)} = v^\star
\]
on $\Xf$. Therefore $v \le v^\star$ for any $v$ feasible in (\ref{opt:LPinf}). We already have proven that $v^\star \in \Bc(\Xf)$; therefore, it only remains to prove that the constraint $v^\star \le \bar l + \alpha v^\star \circ \bar f$ is satisfied. We have
\[
\alpha v^\star\circ \bar f = \sum_{k=0}^{\infty}\alpha^{k+1} \bar l \circ \bar f^{(k+1)} = \sum_{k=1}^{\infty}\alpha^{k} \bar l \circ \bar f^{(k)} = \sum_{k=0}^{\infty}\alpha^{k} \bar l \circ \bar f^{(k)}  - \bar l = v^\star - \bar l.
\]
Therefore, the inequality constraint in~(\ref{opt:LPinf}) is satisfied with equality by $v^\star$. This concludes the proof of optimality of $v^\star$.

\emph{Proof of 1}. If $x \in \Xf_\infty$, then $f(\bar f^{(k)}(x)) \in \Xf$ for all $k\in \Nb_0$ and hence $\bar l (\bar f^{(k)}(x)) = \dist_{\Xf}( f(\bar f^{(k)}(x)) ) = 0$; as a result $v^\star = 0$ on $\Xf_\infty$. Conversely, if $x \in \Xf\setminus \Xf_\infty$, then there exists a $k\in \Nb_0$ such that $\bar f^{(i)}(x) = f^{(i)}(x) \in \Xf$ for all $i \le k$ and $  f(\bar f^{(k)} (x) ) = f^{(k+1)}(x) \notin \Xf   $.  Therefore \[
\bar l(\bar{f}^{(k)} (x)) = \dist_\Xf    \big(f(\bar f^{(k)} (x) ) \big) > 0  
\]
and hence $v^\star > 0$ on $\Xf \setminus \Xf_\infty$.

\emph{Proof of 2}. This follows directly from the previous point.

\emph{Proof of 3}. It suffices to prove continuity of $v^\star$ on $\Xf$. Uniform continuity then follows by compactness of $\Xf$. Let $x \in \Xf$ and $\epsilon > 0$ be given. Then for any $y \in \Xf$ we have from~(\ref{eq:vstar_characterization})
\begin{align*}
|v^\star(x) - v^\star(y)| & \le \sum_{k=0}^\infty\alpha^k | \bar l \big(\bar{f}^{(k)}(x)\big) - \bar l \big(\bar{f}^{(k)}(y)\big)|   \\
  & \le  \sum_{k=0}^{M-1}\alpha^k | \bar l \big(\bar{f}^{(k)}(x)\big) - \bar l \big(\bar{f}^{(k)}(y)\big)| + \sum_{k=M}^{\infty}2 \alpha^k,
\end{align*}
where in the second inequality we used the fact that $\bar l \le 1$ on $\Xf$. The constant $M$ is chosen such that $2\sum_{k=M}^\infty\alpha^k < \epsilon /2$.  Since $\bar l$ and $ \bar f$ are continuous by assumption, so are $\bar l\circ \bar f^{(k)}$ for $k\in\{0,\ldots, M-1\}$. Therefore, there exists $\delta > 0$ such that \[
\sum_{k=0}^{M-1}\alpha^k | \bar l \big(\bar{f}^{(k)}(x)\big) - \bar l \big(\bar{f}^{(k)}(y)\big)| < \frac{\epsilon}{2}
\]
for any $y$ with $\|y - x\|_2 \le \delta$. As a result, $|v^\star(x) - v^\star(y)| < \epsilon$ for all such $y$, as desired.

\emph{Proof of 4}.
Since $\Xf$ is convex, the functions $\mr{proj}_\Xf$ and $\mr{dist}_\Xf$ are Lipschitz continuous with Lipschitz constant one. Hence, the Lipschitz constant of $\bar l \circ \bar f^{(k)}$ is at most $L_f^k$. Therefore, from~(\ref{eq:vstar_characterization}), we have for any $x,y\in\Xf$
\begin{align*}
|v^\star(x) - v^\star(y)| & \le \sum_{k=0}^\infty\alpha^k | \bar l \big(\bar{f}^{(k)}(x)\big) - \bar l \big(\bar{f}^{(k)}(y)\big)|  \le \sum_{k=0}^\infty\alpha^k  L_f^k \|x -y \|_2 = \frac{1}{1 - \alpha L_f} \|x -y \|_2.
\end{align*}
\vglue-8mm
\hfill $\square$

\paragraph{Proof of Lemma~\ref{lem:dp}}
The proof follows closely the argument of~\cite[Theorem~2]{vanRoy2003}; we give full details of the proof since we refer to it later in the paper. Let $v\in \mr{Argmin}_{h\in\mathcal{V}_N}\|v^\star -h\|_{\C(\Xf)}$ (at least one minimizer exists since the set of minimizers is compact\footnote{This follows from the fact that $\|v\|_{\C(\Xf)} \le 2 \|v^\star\|_{\C(\Xf)}$ for any minimizer $v$ and the fact that $\mathcal{V}_N$ is a finite dimensional subspace of $\C(\Xf)$, which implies that the $\C(\Xf)$ norm on $\mathcal{V}_N$ is equivalent to the ``coefficient norm'' given by $\|v\|_{\mr{coef}} = \|c\|_2$ for $v(x) = c^\top \bs\beta(x)$.\label{foot:coefs}}). The idea is to shift~$v$ downward in order to make it satisfy the constraint of~(\ref{opt:LPinf_N}). Set $\bar v_N = v - a$ with $a \in \Rb_+$; then we have
\begin{align*}
\bar v_N- \bar{l} - \alpha \bar v_N\circ \bar{f} &=  v  - \bar{l}   - \alpha v\circ \bar f  - a(1-\alpha) \\
&\le  v^\star  - \bar{l}   - \alpha v^\star\circ \bar f  +\|v-v^\star\|_{\C(\Xf)}(1+\alpha)          - a(1-\alpha) \\
& \le  \|v-v^\star\|_{\C(\Xf)}(1+\alpha)          - a(1-\alpha).
\end{align*}
The last expression is rendered nonnegative by setting $a = \frac{1+\alpha}{1-\alpha}\|v-v^\star\|_{\C(\Xf)}$. Estimating the error for this choice of $a$, writing $\| \cdot\|$ for $\| \cdot\|_{\C(\Xf)}$, yields
\[
\|\bar v_N - v^\star\| \le \|  v - v^\star \| +  \|  \bar v_N - v\|= \|  v - v^\star \| + \|  v - v^\star \|\frac{1+\alpha}{1-\alpha} = \frac{2}{1-\alpha}\|  v - v^\star \|,
\]
which finishes the proof of~(\ref{eq:DP_eq1}). To prove~(\ref{eq:DP_eq2}) it suffices to observe that maximizing the objective of~(\ref{opt:LPinf_N}) is the same as minimizing the left-hand side of~(\ref{eq:DP_eq2}) since $v_N \le v^\star$ for any $v_N$ feasible in~(\ref{opt:LPinf_N}) and hence the absolute value is redundant. \hfill $\square$

\paragraph{Proof of Theorem~\ref{thm:aux}} \emph{Point 1:} This follows from the fact that $v_N$ is feasible in~(\ref{opt:LPinf}) and from Point 2 of Theorem~\ref{thm:mainLP}.

\emph{Point 2:} From Lemma~\ref{lem:dp}, we only need to bound $e_d := \min_{v\in\mathcal{P}_d} \|v^\star -v \|_{\C(\Xf)}$, where $\mathcal{P}_d$ is the space of multivariate polynomials up to degree $d$. From~\cite[Theorem~1]{bagby2002multivariate} and the Kirszbraun extension theorem of Lipschitz functions, we know that $e_d \le C_{\Xf,n}\mr{Lip}(v^\star)\frac{1}{d}$ for some constant $C_{\Xf,n}$ depending only on the diameter of $\Xf$ and the dimension $n$. The result then follows from Point 4 of Theorem~\ref{thm:mainLP}.

\emph{Point 3:} Given $\gamma> 0$, we have 
\begin{align}
\refMeas(\Xf_N\setminus \Xf_\infty) &= \refMeas(v_N \le 0, v^\star > 0  ) \nonumber \\ &= \refMeas(v_N \le 0, v^\star > \gamma  ) +  \refMeas(v_N \le 0, v^\star \in (0,\gamma] ) \nonumber \\ \label{eq:gammaBound}
 &\le \refMeas(v^\star - v_N > \gamma) + \refMeas(v^\star \in (0,\gamma]) \le \frac{1}{\gamma}\int (v^\star-v_N)\,d\lambda + g_{v^\star}(\gamma),
\end{align}
where $g_{v^\star}$ is defined in~(\ref{eq:gvstar_def}). From Point 2 of this theorem we obtain
\[
\refMeas(\Xf_N\setminus \Xf_\infty) \le \frac{1}{\gamma}\frac{2 C_{\Xf,n}}{(1-\alpha)(1-\alpha L_f)} \frac{1}{d} + g_{v^\star}(\gamma).
\]
Minimizing over $\gamma$ gives~(\ref{eq:convRate_opt}); choosing $\gamma = 1/\sqrt{d}$ gives~(\ref{eq:convRate_sqrtd}). \hfill $\square$

\paragraph{Proof of Lemma~\ref{lem:epsNet}} It is enough to prove the statement with $\Xf$ replaced by the smallest axis-aligned box enclosing it. Let therefore $\Xf$ be a box with the largest side of length $D$ and let $(y_i)_{i=1}^{K'}$ be an $\epsilon/2$ net for this box; such net exists with $K' \ge \frac{2^n D^n}{\epsilon^n}$. The probability of sampling a point within $B_{\epsilon/2}(y_i)$ is then $B_{\epsilon/2}(y_i) / \vol(\Xf) = \frac{\epsilon^n}{2^nD^n}$.  Given $K$ independent samples, the probability of not sampling within a given ball is therefore $(1 - \frac{\epsilon^n}{2^nD^n})^K$. Now, by triangle inequality, for any point $x\in B_{\epsilon/2}(y_i)$ it holds that $B_\epsilon(x) \supset B_{\epsilon/2}(y_i)$. Therefore, using the union bound, the probability of \emph{not} obtaining an $\epsilon$ net with $K$ samples is at most
\[
 \frac{2^n D^n}{\epsilon^n} \left(1 - \frac{\epsilon^n}{2^nD^n}\right)^K.
\]
We want to make this quantity less than or equal to $\delta$. Taking logarithm, this is equivalent to
\[
n\log\left(\frac{2D}{\epsilon}\right) + K\log\left(1-\frac{\epsilon^n}{2^n D^n}  \right) \le \log(\delta),
\]
which is equivalent to (\ref{eq:sampNum}). \hfill $\square$

\paragraph{Theorem~\ref{thm:mainBound}}
\emph{Point 1}:  Consider the auxiliary optimization problem
\begin{equation}\label{opt:LP_aux_proof}
\begin{array}{rclll}
&& \sup\limits_{v\in\Vc_N} & \displaystyle    \int_\Xf v(x)\,d\lambda(x)    \vspace{1mm}  \\ 
&& \mathrm{s.t.} & v(x_i)  \le  \dist_\Xf(f(x)) + \alpha v(\proj_\Xf (f(x)))  \:\: &\forall\, x\in\Xf \vspace{1mm} \\
&&& -1 \le v(z_i) \le (1 - \alpha)^{-1}  \:\: &\forall\, i\in 1,\ldots,K',\ 
\end{array}
\end{equation}
which is the LP~(\ref{opt:LPinf_N}) with the additional constraint $-1\le v \le (1-\alpha)^{-1}$ imposed on the artificial data set~(\ref{eq:dataArtif}). By assumption~\ref{as:unisolvent}, the Lipschitz constant of any function $v$ feasible in this problem is bounded by some $L' < \infty$ and the coefficient vector $\cf$ defining $v$ is bounded by some $M'$. Consequently, the supermum~(\ref{opt:LP_aux_proof}) is attained by some $v_N'\in\Vc_N$ which assume to be unique for the simplicity of the argument; we generalize the argument to the case of multiple solutions later. Since $v_N'$ is feasible in~(\ref{opt:LPinf_N}) and hence in~(\ref{opt:LPinf}) it follows by Theorem~\ref{thm:mainLP} that
\[
v_N' \le v^\star \quad \mr{on}\;\; \Xf.
\]
and $\Xf_N' \supset \Xf_N$,
where
\[
\Xf_N' = \{ x \mid v_N'(x) \le 0  \}.
\]

Define the function 
\[
h(\epsilon) = \vol \big( \{x  \mid v_N'(x) \in [-\epsilon,0]\} \big) .
\]
Since $\vol\left(\{x  \mid v_N'(x) = 0    \}\right) = 0$ by Assumption~\ref{ass:basis2}, it follows that 
\[
\lim_{\epsilon \to 0^+}h(\epsilon) =  \vol\left(\{x  \mid v_N'(x) = 0    \}\right) = 0.
\]

Now we prove that $v_{N,K} = c_{N,K}\to\bs\beta$ convergences to $v_N'$ uniformly as $K\to\infty$, where $c_{N,K}$ denotes an optimal solution to~(\ref{opt:LPfinite}). Since the feasible set of (\ref{opt:LPfinite}) is compact it follows that, as $K$ tends to infinity, the sequence $c_{N,K}$ of optimal solutions to~(\ref{opt:LPfinite}) has an accumulation point $\tilde c_N$. Let $\tilde v_N = \tilde c_N^\top \bs\beta(x)$.  By continuity of the functions in the constraints of~(\ref{opt:LP_aux_proof}) and the fact that the points $(x_i)_{i=1}^K$ are sampled independently and uniformly over the compact set $\Xf$ satisfying Assumption~\ref{as:X}, it follows that $\tilde v_N$ is, with probability one, feasible in (\ref{opt:LP_aux_proof}). Therefore $\tilde v_N$ is optimal in~(\ref{opt:LP_aux_proof}) by continuity of the objective function and by the fact that for each $K$ the constraints of the sampled LP~(\ref{opt:LPfinite}) are less stringent than those of (\ref{opt:LP_aux_proof}). Hence $\tilde v_N = v_N'$ on $\Xf$ by the uniqueness of optimizers to~(\ref{opt:LP_aux_proof}). Let $c_{N,K_i}$ be a subsequence such that $\lim_{i\to\infty}c_{N,K_i} = \tilde c_N$. Since $\Vc_N$ is finite-dimensional, convergence in the space of coefficients implies uniform convergence on any compact set\footnote{Given a compact set $\mathbf{K}$, we can define a norm on $\Vc_N$ by $\|\cdot\| := \|\cdot\|_{\C(\mathbf{K})} + \|\cdot\|_{\mr{coef}} $, where $ \|\cdot\|_{\mr{coef}}$ is the coefficient norm as in Footnote~\ref{foot:coefs}. However, since $\Vc_N$ is finite dimensional it follows that this newly defined norm is equivalent to $ \|\cdot\|_{\mr{coef}}$ and in particular that convergence in $ \|\cdot\|_{\mr{coef}}$ implies convergence in $\|\cdot\|_{\C(\mathbf{K})} $. }; therefore $\lim_{i\to\infty} \|v_{N,K_i} - v_N \|_{\C(\Xf)} \to 0$. By the uniqueness of optimizers to~(\ref{opt:LP_aux_proof}), this implies that each accumulation point of $v_{N,K}$ must be equal to $v_N'$ and hence the limit of $v_{N,K}$ exists and is equal to $v_N'$, i.e., 
\[
\lim_{K\to\infty} \|v_{N,K} - v_N \|_{\C(\Xf)} \to 0
\]
with probability one.

To finish the proof, we use the following estimate valid for any $\epsilon > 0$:
\begin{align}
\vol(\Xf_\infty \setminus \Xf_{N,K}) & = \vol(v^\star \le 0,\, v_{N,K} > 0) \nonumber \\
& \le \vol(v_N' \le 0,\, v_{N,K} > 0) \nonumber \\
& = \vol(v_N' < \epsilon,\, v_{N,K} > 0) + \vol(v_N' \in [-\epsilon,0],\, v_{N,K} > 0) \nonumber \\
& \le \vol(v_N' < \epsilon,\, v_{N,K} > 0) + h(\epsilon). \label{eq:aux_in_proof}
\end{align}
Since $v_{N,K}$ converges to $v_N'$ uniformly, it follows that
\[
\lim_{K\to\infty} \vol(\Xf_\infty \setminus \Xf_{N,K}) \le h(\epsilon).
\]
Since $\epsilon$ was arbitrary and $h(\epsilon) \to 0$ as $\epsilon \to 0^+$, the result follows.

To prove the result in the general case where solution to~(\ref{opt:LP_aux_proof}) is not unique, observe that the set of solutions to this LP is compact (since its feasible set is compact). Redefining the function $h$ as
\[
h(\epsilon) = \sup_{v_N'\; \mr{solves}~(\ref{opt:LP_aux_proof})}\vol \big( \{x  \mid v_N'(x) \in [-\epsilon,0]\} \big), 
\]
it follow that $h(\epsilon) \to 0$ as $\epsilon \to 0^+$. Indeed, if $\lim_{\epsilon \to 0^+}h(\epsilon) = \delta > 0$, then by compactness of the set of optimizers of~(\ref{opt:LP_aux_proof}), there would exist an optimal $v_N'$ such that $\lim_{\epsilon\to 0^+}\vol(\{x  \mid v_N'(x) \in [-\epsilon,0]\}) = \delta > 0$, which contradicts Assumption~\ref{ass:basis2}. To finish the proof, consider a subsequence $(K_i)_{i=1}^\infty$ such that
\[
\limsup_{K\to\infty} \vol(\Xf_\infty \setminus \Xf_{N,K}) = \lim_{i\to\infty} \vol(\Xf_\infty \setminus \Xf_{N,K_i}) .
\] 
From this subsequence we can extract another subsequence $(K_{i_j})_{j=1}^\infty$ such that $c_{N,K_{i_j}}$ converges to some $c_N$; this is possible by the compactness of the set of minimizers of~(\ref{opt:LP_aux_proof}). It follows by the same arguments as in the case of a unique minimizer that $c_N$ is optimal and that  $v_{N,K_{i_j}}$ converges as $j\to\infty$ uniformly to an optimal $v_N'$. It follows by the same calculation as~(\ref{eq:aux_in_proof}) that
\[
\lim_{j\to\infty} \vol(\Xf_\infty \setminus \Xf_{N,K_{i_j}}) = 0
\]
and hence 
\[
\limsup_{K\to\infty} \vol(\Xf_\infty \setminus \Xf_{N,K}) = \lim_{i\to\infty} \vol(\Xf_\infty \setminus \Xf_{N,K_i}) = \lim_{j\to\infty} \vol(\Xf_\infty \setminus \Xf_{N,K_{i_j}}) = 0.
\]
as desired.



\emph{Point 2}: 
 With $K$ as in the theorem, the points $(x_i)_{i=1}^K$ form an
\[
\frac{(1-\alpha)\epsilon}{L_{n,N}}\;\; \mr{net}.
\]
Define the function 
\[
\tilde{v}_{N,K} = v_{N,K} - \epsilon.
\]
A straightforward computation shows that $\tilde{v}_{N,K}$ is feasible in (\ref{opt:LPinf_N}). Then we have
\begin{align*}
\refMeas(\Xf_{N,K} \setminus & \;\Xf_\infty) = \refMeas(v_{N,K} \le 0, v^\star > 0) =  \refMeas(v_{N,K} \le 0, v^\star > \gamma) + \refMeas(v_{N,K} \le 0, 0 < v^\star \le \gamma) \\
& \le  \refMeas(v^\star - v_{N,K} > \gamma) + \refMeas(0 < v^\star \le \gamma) \\ & \le \refMeas(v^\star - \tilde v_{N,K} > \gamma + \epsilon) + g_{v^\star}(\gamma) \\
& \le \frac{1}{\gamma + \epsilon}\int v^\star - \tilde v_{N,K} + g^\star(\gamma) = \frac{1}{\gamma + \epsilon}\int v^\star - v_{N,K} + \refMeas(\Xf)\frac{\epsilon}{\gamma + \epsilon}  + g_{v^\star}(\gamma) \\
& = \frac{1}{\gamma + \epsilon}\underbrace{\int v^\star -  v_{N}}_{\le (\ref{eq:bound_vn})}     +   \underbrace{\frac{1}{\gamma + \epsilon}\int v_{N} - v^\star}_{\le 0}  + \frac{1}{\gamma + \epsilon}\underbrace{ \int v^\star - v_{N,K}}_{\le^\star (\ref{eq:bound_vn})} +   \refMeas(\Xf)\frac{\epsilon}{\gamma + \epsilon}  +   g_{v^\star}(\gamma) \\
& \le \frac{4}{\gamma + \epsilon} \frac{C_{\Xf,n}}{(1-\alpha)(1-\alpha L_f)} \frac{1}{d} +  \refMeas(\Xf)\frac{\epsilon}{\gamma + \epsilon}  +   g_{v^\star}(\gamma),
\end{align*}
where the inequality indicated by the brace with a star holds for $ d\ge C_{\Xf,n}(1+\alpha)[(1-\alpha)(1-\alpha L_f)]^{-1}$, as assumed. This is because in this case  \[    C_{\Xf,n}(1+\alpha)[(1-\alpha)(1-\alpha L_f)]^{-1}\min_{v\in\mathcal{P}_d}\| v^\star -v  \|_{\C(\Xf)} \le 1\] and hence the $\bar v_N$ used in the proof of Lemma~\ref{lem:dp} will also satisfy the additional constraint $-1 \le \bar v_N \le (1-\alpha)^{-1}$ of~(\ref{opt:LPfinite}), making the bound (\ref{eq:bound_vn}) valid also for $v_{N,K}$. Setting $\gamma = 1/\sqrt{d}$ gives the desired result.  \hfill $\square$

\paragraph{Proof of Thereom~\ref{thm:mainLP_cont}}
The fact that $v^\star$ is Borel measurable follows from the fact that it is an infimum of a set of Borel measurable functions (indexed by the sequences of control inputs). Boundedness of $v^\star$ is immediate since $\alpha \in (0,1)$ and $|\bar l| \le 1$. The attainment of the supremum in~(\ref{opt:LPinf_cont}) by $v^\star$ follows from the Bellman's optimiality principle. Indeed, this principle stays that
\[
v^\star(x) = \inf_{u\in \Uf} \{\bar l (x,u) + \alpha v^\star \circ \bar f(x,u)    \}
\]
and hence necessarily
\[
v^\star(x) \le \bar l (x,u) + \alpha v^\star \circ \bar f(x,u)
\]
for all $(x,u) \in \Xf \times \Uf$, i.e., $v^\star$ satisfies the constraint of~(\ref{opt:LPinf_cont}). The same argument as in the proof of Theorem~\ref{thm:mainLP} shows that $v \le v^\star$ on $\Xf$ for any $v$ satisfying the constraint of~(\ref{opt:LPinf_cont}). This implies that $v^\star$ is optimal in~(\ref{opt:LPinf_cont}) as claimed.

\emph{Proof of 1 and 2:} Follows by the same arguments as in Theorem~\ref{thm:mainLP}. 

\emph{Proof of 3:} Observe that, by continuity of $\bar f$ and $\bar l$, for every $x_0\in \Xf$ the function 
\[
(u_k)_{k=0}^\infty \mapsto \sum_{k=0}^\infty \alpha^k \bar l (x_k,u_k),
\]
with $x_{k+1} = \bar f(x_k,u_k)$, is continuous with respect to the product topology on $ \Uf^\infty$. Since $\Uf$ is compact, so is $\Uf^\infty$ (e.g., by the Tychonoff's theorem) and hence for every $x_0\in \Xf$ the infimum in (\ref{eq:vstar_characterization_cont}) is attained by some $(u_k)_{k=0}^\infty \in \Uf^\infty$. Let $x_0 \in \Xf$ and $\epsilon > 0$ be given and let $(u_k)_{k=0}^\infty$ be an optimal sequence in~(\ref{eq:vstar_characterization_cont}) associated to this $x_0$. Then for any $x_0' \in \Xf$, with associated optimal sequence $(u_k')_{k=0}^\infty$, we have 
\begin{align*}
v^\star(x_0) - v^\star(x_0') & = \sum_{k=0}^\infty\alpha^k  [\bar l (x_k,u_k) - \bar l (x_k',u_k')]   \\
	& =  \underbrace{\sum_{k=0}^\infty\alpha^k [ \bar l (x_k,u_k) - \bar l(x_k,u_k')]}_{\le\, 0 \; \mr{by\;optimality\;of\;}u_k} +  \sum_{k=0}^\infty\alpha^k [\bar l(x_k,u_k')- \bar l (x_k',u_k')] \\
  & \le  \sum_{k=0}^{M-1}\alpha^k | \bar l (x_k,u_k') - \bar l (x_k',u_k')|+ \sum_{k=M}^{\infty}2 \alpha^k,
\end{align*}
where $M$ is chosen such that $\sum_{k=M}^{\infty}2 \alpha^k <\epsilon /2$. By continuity of $\bar l$ and $\bar f$ and compactness of $\Xf$ and $\Uf$, the function $(x_0,(u_k)_{k=0}^{M-1})\mapsto \sum_{k=0}^\infty \alpha^k \bar l(x_k,u_k)$, where $x_{k+1} = \bar f(x_k,u_k)$, is uniformly continuous on $\Xf \times \Uf^M$. Therefore there exists $\delta > 0$ such that \[ \sum_{k=0}^{M-1}\alpha^k | \bar l (x_k,u_k') - \bar l (x_k',u_k')| \le \epsilon / 2\]whenever $\|x_0 - x_0'\|_2 < \delta$ and hence $v^\star(x_0) - v^\star(x_0') \le \epsilon$ whenever $\|x_0 - x_0'\|_2 < \delta$. A mirror argument shows that $v^\star(x_0') - v^\star(x_0)  < \epsilon$ whenever $\|x_0 - x_0'\|_2 < \delta$ and hence $v^\star$ is continuous. Uniform continuity follows by compactness of $\Xf$.

\emph{Proof of 4:} Using the same notation and reasoning as in the proof of Point 3, we have for any $x_0,  x_0' \in \Xf$ 
\begin{align*}
v^\star(x_0) - v^\star(x_0') \le \sum_{k=0}^\infty\alpha^k [\bar l(x_k,u_k')- \bar l (x_k',u_k')] .
\end{align*}
Since under the stated assumptions, the function $\bar l$ is 1-Lipschitz and $\bar f$ is jointly Lipschitz in $(x,u)$ with Lipschitz  constant $L_f$, it follows that for any $(u_j)_{j=0}^{k-1}$ the function $x_0 \mapsto \bar l(x_k,u_k)$, with $x_{j+1} = \bar f(x_j,u_j)$, has Lipschitz constant at most $L_f^k$. As a result,
\[
v^\star(x_0) - v^\star(x_0') \le  \sum_{k=0}^\infty\alpha^k L_f^k = \frac{1}{1 - \alpha L_f}.
\]
A mirror argument shows that also $v^\star(x_0') - v^\star(x_0) \le (1 - \alpha L_f)^{-1} $, finishing the proof. \hfill $\square$

\bibliographystyle{abbrv}
\bibliography{References}

\begin{thebibliography}{10}

\bibitem{bagby2002multivariate}
T.~Bagby, L.~Bos, and N.~Levenberg.
\newblock Multivariate simultaneous approximation.
\newblock {\em Constructive approximation}, 18(4):569--577, 2002.

\bibitem{Blanchini_paper}
F.~Blanchini.
\newblock Set invariance in control.
\newblock {\em Automatica}, 35(11):1747--1767, 1999.

\bibitem{BlanchiniBook}
F.~Blanchini and S.~Miani.
\newblock {\em Set-Theoretic Methods in Control}.
\newblock Birkh{\"a}user Boston, first edition, 2007.

\bibitem{ciarlet2002finite}
P.~G. Ciarlet.
\newblock {\em The finite element method for elliptic problems}, volume~40.
\newblock {SIAM}, 2002.

\bibitem{vanRoy2003}
D.~P. De~Farias and B.~Van~Roy.
\newblock The linear programming approach to approximate dynamic programming.
\newblock {\em Operations research}, 51(6):850--865, 2003.

\bibitem{federer1959curvature}
H.~Federer.
\newblock Curvature measures.
\newblock {\em Transactions of the American Mathematical Society},
  93(3):418--491, 1959.

\bibitem{goluskin2018bounding}
D.~Goluskin.
\newblock Bounding extreme values on attractors using sum-of-squares
  optimization, with application to the lorenz attractor.
\newblock {\em arXiv preprint arXiv:1807.09814}, 2018.

\bibitem{gondhalekar_controlled_2009}
R.~Gondhalekar, J.~Imura, and K.~Kashima.
\newblock Controlled invariant feasibility -- {A} general approach to enforcing
  strong feasibility in {MPC} applied to move-blocking.
\newblock {\em Automatica}, 45(12):2869--2875, 2009.

\bibitem{guide2006infinite}
A.~H. Guide.
\newblock {\em Infinite dimensional analysis}.
\newblock Springer, 2006.

\bibitem{gutman1987algorithm}
P.-O. Gutman and M.~Cwikel.
\newblock An algorithm to find maximal state constraint sets for discrete-time
  linear dynamical systems with bounded controls and states.
\newblock {\em IEEE Transactions on Automatic Control}, 32(3):251--254, 1987.

\bibitem{kordaROA}
D.~Henrion and M.~Korda.
\newblock Convex computation of the region of attraction of polynomial control
  systems.
\newblock {\em IEEE Transactions on Automatic Control}, 59(2):297--312, 2014.

\bibitem{HernandezLasserre}
O.~Hernandez-Lerma and J.~B. Lasserre.
\newblock {\em Discrete-Time Markov Control Processes: Basic Optimality
  Criteria}.
\newblock Springer, first edition, 1996.

\bibitem{johansson1997computation}
M.~Johansson and A.~Rantzer.
\newblock Computation of piecewise quadratic lyapunov functions for hybrid
  systems.
\newblock In {\em 1997 European Control Conference (ECC)}, pages 2005--2010.
  IEEE, 1997.

\bibitem{kenanian2019data}
J.~Kenanian, A.~Balkan, R.~M. Jungers, and P.~Tabuada.
\newblock Data driven stability analysis of black-box switched linear systems.
\newblock {\em Automatica}, 109:108533, 2019.

\bibitem{me_nolcos}
M.~Korda, D.~Henrion, and C.~N. Jones.
\newblock Inner approximations of the region of attraction for polynomial
  dynamical systems.
\newblock In {\em IFAC Symposium on Nonlinear Control Systems (NOLCOS)}, 2013.

\bibitem{kordaMCI}
M.~Korda, D.~Henrion, and C.~N. Jones.
\newblock Convex computation of the maximum controlled invariant set for
  polynomial control systems.
\newblock {\em SIAM Journal on Control and Optimization}, 52(5):2944--2969,
  2014.

\bibitem{korda2018pdes}
M.~Korda, D.~Henrion, and J.-B. Lasserre.
\newblock Moments and convex optimization for analysis and control of nonlinear
  partial differential equations.
\newblock {\em arXiv preprint arXiv:1804.07565}, 2018.

\bibitem{korda2018invarMeasure}
M.~Korda, D.~Henrion, and I.~Mezic.
\newblock Convex computation of extremal invariant measures of nonlinear
  dynamical systems and markov processes.
\newblock {\em arXiv preprint arXiv:1807.08956}, 2018.

\bibitem{lasserre2008nonlinear}
J.~B. Lasserre, D.~Henrion, C.~Prieur, and E.~Tr{\'e}lat.
\newblock Nonlinear optimal control via occupation measures and
  {LMI}-relaxations.
\newblock {\em SIAM Journal on Control and Optimization}, 47(4):1643--1666,
  2008.

\bibitem{lewis1980relaxation}
R.~Lewis and R.~Vinter.
\newblock Relaxation of optimal control problems to equivalent convex programs.
\newblock {\em Journal of Mathematical Analysis and Applications},
  74(2):475--493, 1980.

\bibitem{yalmip}
J.~L{\" o}fberg.
\newblock Yalmip : A toolbox for modeling and optimization in {MATLAB}.
\newblock In {\em Proceedings of the CACSD Conference}, Taipei, Taiwan, 2004.

\bibitem{magron2018semidefinite}
V.~Magron, M.~Forets, and D.~Henrion.
\newblock Semidefinite approximations of invariant measures for polynomial
  systems.
\newblock {\em arXiv preprint arXiv:1807.00754}, 2018.

\bibitem{marx2018moment}
S.~Marx, T.~Weisser, D.~Henrion, and J.~Lasserre.
\newblock A moment approach for entropy solutions to nonlinear hyperbolic pdes.
\newblock {\em arXiv preprint arXiv:1807.02306}, 2018.

\bibitem{mayne_constrained_2000}
D.~Q. Mayne, J.~B. Rawlings, C.~V. Rao, and P.~O.~M. Scokaert.
\newblock Constrained model predictive control: Stability and optimality.
\newblock {\em Automatica}, 36(6):789--814, 2000.

\bibitem{oustry2019inner}
A.~Oustry, M.~Tacchi, and D.~Henrion.
\newblock Inner approximations of the maximal positively invariant set for
  polynomial dynamical systems.
\newblock {\em IEEE Control Systems Letters}, 3(3):733--738, 2019.

\bibitem{prajna2004nonlinear}
S.~Prajna, P.~Parrilo, A.~Rantzer, et~al.
\newblock Nonlinear control synthesis by convex optimization.
\newblock {\em Automatic Control, IEEE Transactions on}, 49(2):310--314, 2004.

\bibitem{rantzer2001dual}
A.~Rantzer.
\newblock A dual to lyapunov's stability theorem.
\newblock {\em Systems \& Control Letters}, 42(3):161--168, 2001.

\bibitem{rubio1975generalized}
J.~Rubio.
\newblock Generalized curves and extremal points.
\newblock {\em SIAM Journal on Control}, 13(1):28--47, 1975.

\bibitem{rubio1976extremal}
J.~Rubio.
\newblock Extremal points and optimal control theory.
\newblock {\em Annali di Matematica Pura ed Applicata}, 109(1):165--176, 1976.

\bibitem{takens1981detecting}
F.~Takens.
\newblock Detecting strange attractors in turbulence.
\newblock In {\em Dynamical systems and turbulence, Warwick 1980}, pages
  366--381. Springer, 1981.

\bibitem{vinter1978equivalence}
R.~B. Vinter and R.~M. Lewis.
\newblock The equivalence of strong and weak formulations for certain problems
  in optimal control.
\newblock {\em SIAM Journal on Control and Optimization}, 16(4):546--570, 1978.

\bibitem{young1937generalized}
L.~C. Young.
\newblock Generalized curves and the existence of an attained absolute minimum
  in the calculus of variations.
\newblock {\em Comptes Rendus de la Societe des Sci. et des Lettres de
  Varsovie}, 30:212--234, 1937.

\end{thebibliography}

\end{document}